\documentclass[a4paper,11pt]{article}

\usepackage{amsmath, amssymb, amsthm}
\usepackage{graphics}
\usepackage{ulem}      
\usepackage{psfrag, booktabs}
\usepackage{color, epsfig, float}
\usepackage[active]{srcltx}
\usepackage[bookmarksopen, colorlinks, linkcolor = blue, urlcolor = red,
            citecolor = red, menucolor = blue]{hyperref}

\usepackage{float}
\floatstyle{plain}
\newfloat{algorithm}{thp}{alg}
\floatname{algorithm}{Algorithm}

\newtheorem{theorem}{Theorem}[section]
\newtheorem{lemma}[theorem]{Lemma}
\newtheorem{corol}[theorem]{Corollary}
\newtheorem{proposition}[theorem]{Proposition}
\newtheorem{example}[theorem]{Example}

\newcommand{\ipl}{\langle}
\newcommand{\ipr}{\rangle}

\DeclareMathOperator{\Proj}{\mathbf{P}}
\newcommand\R{\mathbb{R}}

\newcommand\N{\mathbb{N}}

\DeclareMathOperator{\argmin}{arg\, min}

\newcommand\norm[1]{\|#1\|}

\newcommand\set[1]{\{#1\}}

\begin{document}


\title{
%
%
%
       Range-relaxed criteria for the choosing Lagrange multipliers
       in nonstationary iterated Tikhonov method}
%

\setcounter{footnote}{1}

\author{
R.~Boiger%
\thanks{Materials Center Leoben Forschung Gmbh, Roseggerstraße 12,
        8700 Leoben, Austria,
        \href{mailto:romana.boiger@mcl.at}{\tt romana.boiger@mcl.at};
        former: Insitut f\"ur Mathematik, Alpen-Adria Universit\"at Klagenfurt,
        Universit\"atsstrasse 65-67, 9020 Klagenfurt, Austria.}
\and
A.~Leit\~ao%
\thanks{Department of Mathematics, Federal University of St.\,Catarina,
        P.O.\,Box 476, 88040-900 Florian\'opolis, Brazil,
        \href{mailto:acgleitao@gmail.com}{\tt acgleitao@gmail.com}.}
\and
B.~F.~Svaiter%
\thanks{IMPA, Estr.\,Dona Castorina 110, 22460-320 Rio de Janeiro, Brazil,
       \href{mailto:benar@impa.br}{\tt benar@impa.br}.} }

\date{\small \today}

\maketitle

\begin{abstract}
In this article we propose a novel {\em nonstationary iterated Tikhonov} (NIT)
type method for obtaining stable approximate solutions to ill-posed operator
equations modeled by linear operators acting between Hilbert spaces.
Geometrical properties of the problem are
used to derive a new strategy for choosing
the sequence of regularization parameters (Lagrange multipliers) for the NIT
iteration.
Convergence analysis for this new method is provided.
Numerical experiments are presented for two distinct applications: I) A 2D
elliptic parameter identification problem (Inverse Potential Problem);
II) An image deblurring problem.
The results obtained validate the efficiency of our method compared with
standard implementations of the NIT method (where a geometrical choice
is typically used for the sequence of Lagrange multipliers).
\end{abstract}

\medskip


\section{Introduction} \label{sec:intro}

In this article we propose a new {\em nonstationary Iterated Tikhonov} (NIT)
type method~\cite[Sec.~1.2]{BS87} for obtaining stable approximations of linear
ill-posed problems.
The Lagrange multiplier is chosen so as to guarantee the residual of
the next iterate to be in a \emph{range}.
Previous strategies for choosing the Lagrange multiplier in each iteration
of NIT type methods either prescribe (\emph{a priori}) a geometrical increase
of this multiplier \cite{HG98} or require (\emph{a posteriori}) the residual
at the next iterate to assume a prescribed value which depends on the current
residual.

In those NIT methods that prescribe a geometrical increase of the Lagrange
multipliers, the use of a too large geometric factor may lead to numerical
instabilities and failure of convergence, whereas the use of a too small
factor leads to a slow convergent method (see Figures~\ref{fig:compare-sIT-nIT}
and~\ref{fig:nIT-unstable}); these features are highly dependent on the problem
at hand and, in general, it is not clear how to adequately choose the geometric
factor.

In those NIT methods that require the residual at the next iterate to assume
a prescribed value, at each iteration one needs to solve a nonlinear equation
which involves the resolvent of an ill-posed operator \cite{Ha97, DoHa13}.
This is accomplished by means of iterative methods (\emph{e.g.}\ Newton)
which require, at each of their steps, the solution of a linear system
for a different operator.
Consequently, the number of iterations required for these methods
does not fully quantify their computational costs.

The main contribution in this article is the proposal of a novel
\emph{a posteriori} strategy for choosing the  Lagrange multipliers
in NIT methods.
Since  we prescribe the residual of the next iterate to be in a range,
the set of feasible Lagrange multipliers, at each iteration,
is a \emph{non-degenerate interval},  which renders feasible their economical
computation (as explained later on).
Many relevant theoretical convergence properties present at previous
\emph{a posteriori} methods~\cite{EngHanNeu96} (e.g., residual
convergence rates, stability, semi-convergence) still hold for our novel strategy.
We also explore the feature of a feasible interval for the Lagrange multiplier
to speed up a Newton-like method for computing it.
The resulting method proves, in our preliminary numerical experiments,
to be more efficient {\color{black} (with respect to computational cost)}
than the geometrical choice of the Lagrange multipliers
\cite{HG98}, typically used in implementations of NIT type methods,
for low noise levels.

The \textit{inverse problem} we are interested in consists of determining an
unknown quantity $x \in X$ from the set of data $y \in Y$, where $X$ and $Y$ are
Hilbert spaces {\color{black} with norms $\norm{\cdot}_X$ and $\norm{\cdot}_Y$ respectively}.
In practical situations, one does not know the data exactly;
instead, only approximate measured data $y^\delta \in Y$ are available with
\begin{equation} \label{eq:noisy-data}
\norm{ y^\delta - y }_Y \ \le \ \delta \, ,
\end{equation}
where $\delta > 0$ is the (known) noise level. The available data $y^\delta$
are obtained by indirect measurements of the parameter $x$, this process
being described by the ill-posed operator equation
\begin{equation} \label{eq:ip}
A \, x \ = \ y \, ,
\end{equation}
where $A: X \to Y$ is a bounded linear operator, whose inverse $A^{-1}: R(A) \to X$
either does not exist, or is not continuous.
Consequently, approximate solutions are extremely sensitive to noise in the data.

Linear ill-posed problems are commonly found in applications ranging from image
analysis to parameter identification in mathematical models.
There is a vast literature on iterative methods for the stable solution of \eqref{eq:ip}.
We refer the reader to the books
\cite{Gr84, Hof86, Bau87, Lo89, Mor93, BakKok04, Kir96, EngHanNeu96, NatWue01, KalNeuSch08}
and the references therein.
Iterated Tikhonov (IT) type methods for solving the ill-posed problem \eqref{eq:ip}
are defined by {\color{black} an} iteration formula
\begin{equation*}
  x_k^\delta \ = \
  \argmin_{x\in X} \big\{ \lambda_k {\color{black} \| A x - y^\delta \|_Y^2}
                          + {\color{black} \| x - x_{k-1}^\delta \|_X^2} \big\} \, ,
\end{equation*}
{\color{black} that} corresponds to
\begin{align} \label{eq:it}
  \begin{aligned}
  x_k^\delta
  &= x_{k-1}^{\delta} -
  (I + \lambda_k A^*A)^{-1} \, \lambda_k A^* (A x_{k-1}^\delta - y^\delta ) \, ,\\
    &=
      (\lambda_k^{-1}I + A^*A)^{-1}     \left[\lambda_k^{-1}x_{k-1}^{\delta}+
      A^* y^\delta  \right]
  \end{aligned}
\end{align}
where $A^*: Y \to X$ is the adjoint operator to $A$.
The parameter $\lambda_k > 0$ can be viewed as the Lagrange multiplier of the problem
of projecting $x^\delta_{k-1}$ onto a levelset of $\norm{Ax - y^\delta}^2$.
If the sequence $\{ \lambda_k = \lambda \}$ is constant, iteration \eqref{eq:it}
is called {\em stationary} IT (or SIT), otherwise it is denominated {\em nonstationary}
IT (or NIT).
{\color{black} To simplify the notation, from now on we will use $\norm{\cdot}$
instead of $\norm{\cdot}_X$ or $\norm{\cdot}_Y$, whenever the norm under consideration
is clearly understood.}

In the NIT methods, each $\lambda_k$ is either chosen \emph{a priori} (e.g., in geometric
progression) or it is chosen \emph{a posteriori}.
In the  \emph{a posteriori} variants, $\lambda_k$ is chosen so that the next
iterate has a prescribed residual which is either a fixed fraction of the
current residual or a fraction which depends also on the noise level
\cite[eq.\ 2.11b in Alg.\ 1]{DoHa13}.
In other words, $\lambda_k$ is chosen in~\eqref{eq:it} so that
\begin{align}
  \label{eq:it.e}
  \norm{Ax_k^\delta-y^\delta}=\Phi(\norm{Ax_{k-1}^\delta-y^\delta},\delta)
\end{align}
(see \cite{Do12} for yet another strategy). We propose $\lambda_k$ to be {\color{black} chosen} so that
\begin{align}
  \label{eq:it.e2}
  \delta\leq   \norm{Ax_k^\delta-y^\delta}\leq\Psi(\norm{Ax_{k-1}^\delta-y^\delta},\delta).
\end{align}
{\color{black}
The upper bound in \eqref{eq:it.e2} for the residual depends on the current residual
and the noise level as in \cite{DoHa13}; however we propose here a new formula
to define this upper bound (see eq.\ \eqref{eq:p-proj}).}

The SIT method for solving \eqref{eq:ip} was considered in \cite{Lo89, Gr84},
where a well developed convergence analysis can be found
{\color{black} (see also Lardy \cite{Lar75}, where the particular choice $\lambda_k = 1$ is analyzed)}.
It is worth distinguishing between the SIT method and the \emph{iterated Tikhonov
methods of order $n$} \cite{KiCh79, Eng87, Sch93a}, where the number of iterations
(namely $n$) is fixed.
In this case $\lambda_k = \lambda > 0$, $k=0,\dots,n-1$, and $\lambda$ plays the
role of the regularization parameter.

The NIT method was addressed by many authors, e.g. \cite{Fak81, HG98, BS87}.
In numerical implementations of this method, the geometrical choice $\lambda_k = q^k$,
$q>1$, is a commonly adopted strategy and we shall refer to the resulting
method as gNIT method ({\em geometrical nonstationary IT method}).

The numerical performances of NIT type methods are superior to the ones
of SIT type methods
{\color{black} in terms of number of iterations and computational time
required to attain a predefined accuracy.}
This fact is illustrated by Example \ref{ex:gNIT-sIT}.

\begin{example} \label{ex:gNIT-sIT}
A linear system modeled by the Hilbert matrix $H^{25 \times 25}$ is
considered in Figure~\ref{fig:compare-sIT-nIT} (random noise of level
$\delta = 0.001\%$ is used).
{\color{black} In this benchmark problem the SIT method is tested with
$\lambda_k = 2$ (RED), while the gNIT method is tested with
$\lambda_k = 2^k$ (BLACK) and $\lambda_k = 3^k$ (BLUE).
In all these tests the same linear solver was used.
}
\end{example}

The computation $\lambda_k = q^k$ in the  gNIT is straightforward; however,
the choice of $q > 1$ is \emph{exogenous} to \eqref{eq:ip},
\eqref{eq:noisy-data} and it is not clear which are the good values for $q$.
Indeed, as shown in the next example, increasing the constant $q$ may lead
either to faster convergence or failure to converge:

\begin{example} \label{ex:gNIT-unstable}
We set $A = H^{25 \times 25}$, $X = Y = \R^{25}$ and random noisy data with
$\delta = 10^{-5}\,\%$.
{\color{black} In Figure~\ref{fig:nIT-unstable} the gNIT method is tested
with $\lambda_k = 2^k$ (BLACK), $\lambda_k = 3^k$ (BLUE) and
$\lambda_k = 4^k$ (ORANGE); in the last test there is not convergence.
}
\end{example}

The above described issues motivated  the use of \emph{a posteriori} choices
for the Lagrange multipliers, which requires the residual at the next iterate
to assume a prescribed value dependent on the current residual and
also on the noise level.
\medskip

\begin{figure}[t]
\centerline{\includegraphics[width=\textwidth]{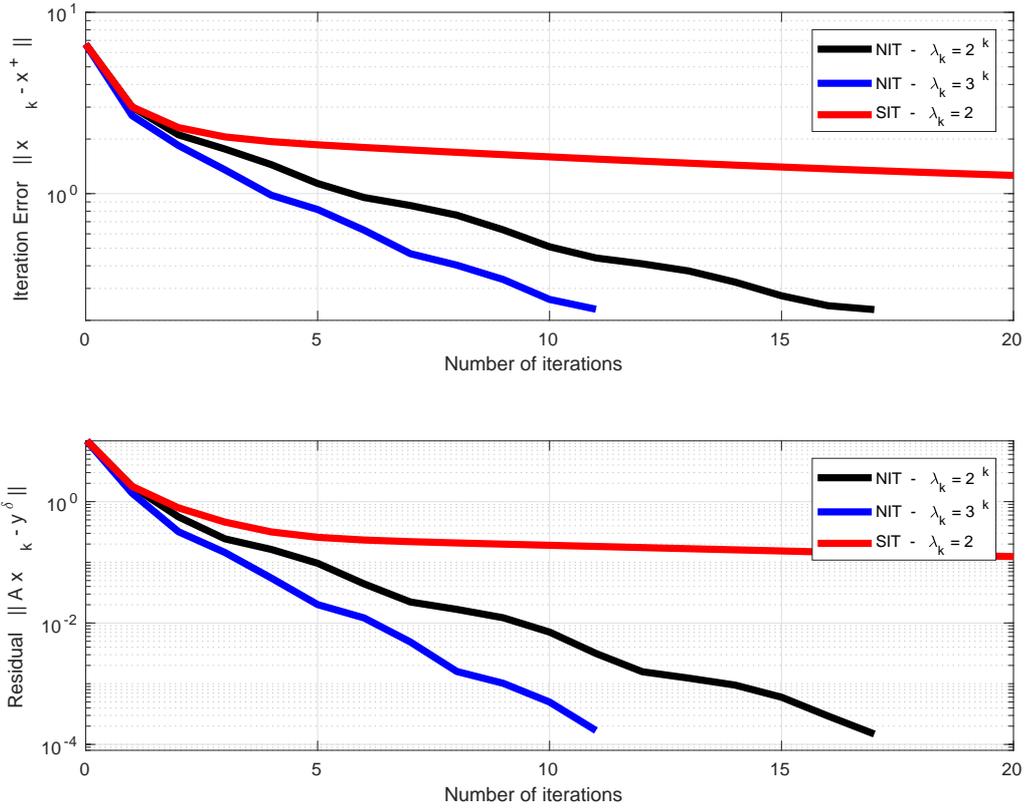}}
\vskip-0.5cm
\caption{\small Comparison between SIT and NIT type methods:
$A = H^{25 \times 25}$ is the square Hilbert matrix; artificial random noise
($\delta = 10^{-3} \%$) is added to the data. The stopping criteria (discrepancy
principle with $\tau = 2$) is reached after 17 steps (BLACK), 11 steps (BLUE)
and 59275 steps (RED).}
\label{fig:compare-sIT-nIT}
\end{figure}

\begin{figure}[t]
\centerline{\includegraphics[width=\textwidth]{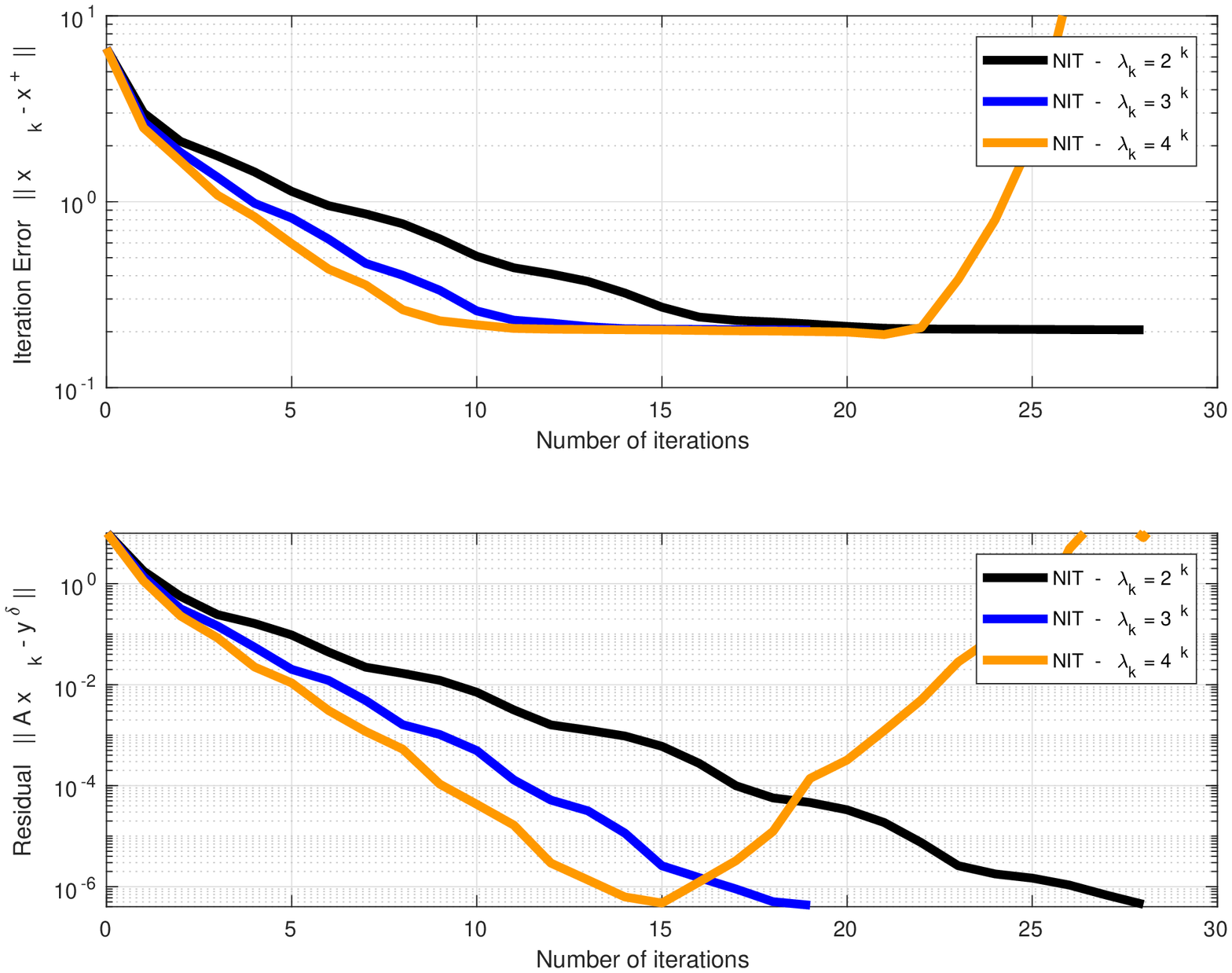}}
\vskip-0.5cm
\caption{\small Unstable behavior of gNIT type methods:
matrix $A \in \R^{25 \times 25}$ as in Figure~\ref{fig:compare-sIT-nIT};
artificial random noise ($\delta = 10^{-5} \%$) is added to the data.
The stopping criteria (discrepancy principle with $\tau = 2$) is reached
after 27 steps (BLACK), 18 steps (BLUE). In the last implementation (ORANGE),
the gNIT method becomes unstable before the stopping criteria is reached.}
\label{fig:nIT-unstable}
\end{figure}

\noindent
Next we briefly review some relevant convergence results for IT type methods.

$\bullet$ \ In \cite{BS87} Brill and Schock proved that, in the exact data case, the
NIT method \eqref{eq:it} converges to a solution of $Ax = y$ if and only if
$\sum \lambda_k = \infty$.
Moreover, a convergence rate result was established under the additional assumption
$\sum \lambda_k^2 < \infty$.

$\bullet$ \ The assumptions needed in \cite{BS87} in order to derive convergence
rates are neither satisfied for the sIT, nor for the NIT with the geometrical choice
$\lambda_k = q^k$, $q > 1$.

$\bullet$ \ In \cite{HG98} rates of convergence are established for the stationary Lardy's
method \cite{Lar75} as well as for the NIT with geometrical choice of $\lambda_k$.
Under the source condition {\color{black} $x^\dag \in \mathop\mathrm{Range}((A^*A)^\nu)$},
where $x^\dag = A^\dag y$ is the normal solution of $Ax = y$%
\footnote{\color{black} I.e., $x^\dag$ is the unique vector satisfying $x^\dag \in
\mathop\mathrm{Dom}(A) \cap \mathop\mathrm{Ker}(A)^\perp$
and $A x^\dag = y$, where $\mathop\mathrm{Dom}(A)$ and $\mathop\mathrm{Ker}(A)$
stand for the domain and the kernel of $A$ respectively.},
and $\nu > 0$,%
\footnote{See \cite[Theorem 2.1]{HG98} for details on the positive the scalar $\nu$.}
the linear rate of convergence $\norm{x_k - x^\dag} = O(q^{k\nu})$ is proven.
\medskip

\noindent
The article is organized as follows:
In Section~\ref{sec:rrNIT} we introduce the new method (rrNIT), which is proposed and
analyzed in this manuscript.
Moreover, a detailed formulation of this method is given, and some preliminary estimates
{\color{black} for the Lagrange multipliers $\lambda_k$} are also derived.
In Section~\ref{sec:conv-anal} a convergence analysis of the rrNIT method is presented.
In Section~\ref{sec:num-implementation} we discuss the algorithmic implementation of the
rrNIT method. In particular, we address the challenging numerical issue of efficiently
computing the Lagrange multipliers $\lambda_k$.
Section~\ref{sec:num-experiments} is devoted to numerical experiments. The
{\em Image deblurring} problem and the {\em Inverse Potential
Problem} are considered in Subsections~\ref{ssec:num-id} and~\ref{ssec:num-ipp}
respectively.
Section~\ref{sec:conclusion} is dedicated to final remarks and conclusions.


\section{Range-relaxed non-stationary iterated Tikhonov method} \label{sec:rrNIT}

In this section first we propose in Subsection~\ref{sec:orrpm} a
conceptual projection method for solving~\eqref{eq:noisy-data},
\eqref{eq:ip}.
In this method, each iterate is obtained projecting the previous one
onto a level set of the residual function. The level set is
prescribed to belong to a range of level sets, instead of being a
single one.
Second, we propose in Subsection~\ref{ssec:true-alg} an implementable
version of the conceptual method where the projection is computed via
Lagrange multipliers.
Finally, we derive some basic properties of the new proposed method.

The implementable method proposed here happens to be a new NIT method
where, in each iteration, the set of feasible choices for the Lagrange
multipliers is an interval, instead of a single real
number. For this reason, we call the new method a range-relaxed Non-stationary
Iterative Tikhonov Method (rrNIT).


For the remaining of this article we suppose that the following assumptions
hold true:
\bigskip

\hfill\begin{minipage}{15cm}
{\bf (A1)} There {\color{black} exists} $x^\star \in X$ such that $A x^\star = y$, where
$y \in \mathop\mathrm{Range}(A)$ are the exact data.
\bigskip

{\bf (A2)} The operator $A: X \to Y$ is linear, bounded and ill-posed, i.e., even if the
operator $A^{-1}: R(A) \to X$ (the left inverse of $A$) exists, it is not continuous.
\end{minipage}

\subsection{A  Successive Orthogonal range-relaxed Projections Method}
\label{sec:orrpm}

We use the notation $\Omega_\mu$, for $\mu\geq 0$, to denote the $\mu$-levelset of
the residual functional $\norm{Ax-y^\delta}$, that is,
\begin{align} \label{eq:best-set}
 \Omega_\mu \ := \ \{x \in X \, : \, \norm{Ax - y^\delta}^2 \leq \mu^2\} \, .
\end{align}
The basic geometric properties of the levelsets $\Omega_\mu$, described next, are
instrumental in the forthcoming analysis.

\begin{proposition} 
  \label{pr:omega-mu}
  Let $\Omega_\mu$ be as in \eqref{eq:best-set}.
  \begin{enumerate}
  \item   For each $\mu \geq 0$, the set $\Omega_\mu$  is
    closed and convex. 
  \item 
    If $\mu'\geq\mu>0$
    then $\Omega_\mu \subset \Omega_{\mu'}$.
  \item  If
    $\mu\geq\delta$ then $A^{-1}(y) \subset \Omega_\mu$.
  \item 
    If  $\mu > \delta$
  then  $\Omega_\mu$ has non-empty interior.
  \end{enumerate}
\end{proposition}

\begin{proof}
{\color{black}
Item 1. follows from the continuity and convexity of $\norm{Ax-y^\delta}$ as a function
of $x$ together with definition \eqref{eq:best-set}.
Item 2. follows trivially from \eqref{eq:best-set}.
Item 3. follows from \eqref{eq:best-set} and assumption \eqref{eq:noisy-data}, because
$\norm{Ax^\star - y^\delta} = \norm{y - y^\delta} \leq \delta$.
The last item follows from this inequality together with definition \eqref{eq:best-set}
and the continuity of $\norm{Ax - y^\delta}$.}
\end{proof}

Notice that all available information about the solution set
$A^{-1}(y)$ is contained in \eqref{eq:noisy-data}, \eqref{eq:ip}.
Thus, in the absence of additional information, $\Omega_\delta$ is the
set of best possible approximate solutions for the inverse problem under
consideration.%
\footnote{I.e., given two elements in $\Omega_\delta$, it is not
  possible to distinguish which of them better approximates $x^\star$.}

Nevertheless the levelset $\Omega_\delta$ is, in general, unbounded and it is
desirable to exclude those approximate solutions with ``too large''
norms.
Moreover, very often a crude estimation $\hat x$ to the
solution of \eqref{eq:ip} is available.  In this context, it is
natural to consider the projection problem
\begin{align} \label{eq:pr-pb}
\left\{
 \begin{array}{ll} 
   \min_x          & \norm{x - \hat x}^2 \\
   \mathrm{s.t.} & \norm{A x - y^\delta}^2 \, \leq \, \mu^2 ,
 \end{array} \right.
\end{align}
where $\mu \geq 0$.
Observe that if $\norm{A\hat x-y^\delta}>\mu\geq\delta$,
then the solution of this projection problem is closer to $\Omega_\delta$
than $\hat x$ and has a smaller residual than $\hat x$.

The considerations in the preceding paragraph show that it is possible,
at least conceptually, to devise projection methods for solving our
ill-posed problem.
Let us briefly discuss the
{\color{black} conditioning of the projection problem} \eqref{eq:pr-pb} 
with respect to {\color{black} the} parameter $\mu$:
\begin{enumerate}
\item[(i)] for $0 \leq \mu < \delta$ {\color{black} the projection problem}
  \eqref{eq:pr-pb} may be unfeasible, that is, it may become the
  problem of projecting $\hat x$ onto an \emph{empty set};
\item[(ii)] for $\mu = \delta$, in view of (i), problem
  \eqref{eq:pr-pb} is in general ill-posed with respect to the
  parameter $\mu$;
\item[(iii)] for $\mu > \delta$ problem
  \eqref{eq:pr-pb} is well posed, and it is natural to expect that
  the larger the $\mu$ the better conditioned it becomes.
\end{enumerate}
%


A compromise between reducing the residual norm $\norm{Ax-y^\delta}$ and
preventing ill-posedness of the projection problem would be to choose
\[
\hat\mu \ = \ p\norm{A\hat x-y^\delta} + (1-p)\delta \, ,
\]
where $0 < p < 1$ quantify this compromise. However,
\begin{enumerate}
\item projecting $\hat x$ onto $\Omega_{\hat\mu}$, which is a \emph{pre-defined}
level set of $\norm{Ax-y^\delta}^2$, entails an additional numerical difficulty:
the projection has to be computed by solving a linear system where the
Lagrange multiplier is implicitly defined by an \emph{algebraic equation};
\item the projection of $\hat x$ onto any levelset $\Omega_{\mu}$ with
$\delta \leq \mu \leq \hat\mu$ is as good as (or even better than) the
projection of $\hat x$ onto $\Omega_{\hat\mu}$.
\end{enumerate}
In view of these observations,  we shall generate $x_k^\delta$ from
$\hat x = x_{k-1}^\delta$ by projecting it onto any one of the range
of convex sets $(\Omega_{\mu})_{\delta\leq\mu\leq\hat\mu}$,
that is, by solving the \emph{range-relaxed} projection problem of
computing $(x,\mu)$ such that
\begin{align} \label{eq:p-proj}
\left\{
  \begin{array}{ll} 
    \min_x          & \norm{x - \hat x}^2 \\
    \mathrm{s.t.} & ||Ax - y^\delta||^2 \, \leq \, \mu^2 \, , \qquad
    \delta \, \leq \, \mu \, \leq \, p\norm{A\hat x - y^\delta} + (1-p)\delta \, ,
  \end{array}
\right.
\end{align}
whenever $x_{k-1}^\delta \notin \Omega_\delta$. 
Observe that this problem has multiple solutions.
The advantage of this strategy is that the set of feasible
Lagrange multipliers of the above problem is an interval with
non-empty interior, as we will discuss latter, instead of a
single point.

In what follows we use the notation ${\Proj}_\Omega$ to denote the orthogonal
projection onto $\Omega$, for $\emptyset\neq\Omega\subset X$ closed and convex.
The discussion in the previous paragraph leads us to propose the
\emph{conceptual} successive orthogonal range-relaxed projection
Method for problem \eqref{eq:noisy-data}, \eqref{eq:ip} described in
Algorithm~\ref{alg:rrNIT-concept}.

\begin{algorithm}[b]
\begin{center}
\fbox{\parbox{13.3cm}{
$[1]$ choose an initial guess \ $x_0 \in X$;
\smallskip

$[2]$ choose \ $p \in (0,1)$ \ and \ $\tau > 1$;
\smallskip

$[3]$ for \ $k \geq 1$ \ do
\smallskip

\ \ \ \ $[3.1]$ compute\, $(x_{k}^\delta,\mu_k)$,

\smallskip

\ \ \ \ \ \ \ \ \ \ 
$x_{k}^\delta={\Proj}_{\Omega_{\mu_k}}(x_{k-1}^\delta)$,
%
\ $\delta \, \leq \, \mu_k \, \leq \,
                      p\norm{Ax_{k-1}^\delta - y^\delta} + (1-p) \delta$;
\smallskip

\ \ \ \ $[3.2]$ stop to iterate at step\, $k^* \geq 1$\, s.t.\,
        $\norm{Ax_{k^*}^\delta - y^\delta} < \tau\delta$ \ for the first time.
} }
\end{center} \vskip-0.5cm
\caption{Successive orthogonal range-relaxed projection method.} \label{alg:rrNIT-concept}
\end{algorithm}

\noindent
Since $\norm{Ax_k^\delta-y^\delta}=\mu_k$, the variable $\mu_k$ is
redundant, nevertheless, its use in the conceptual Algorithm~\ref{alg:rrNIT-concept}
{\color{black} clarifies} the kind of projection problem used to compute $x_k^\delta$.

By its definition, Algorithm 1 is a method of successive orthogonal
projections onto level sets of $\norm{Ax-y^\delta}$. As the level set
used in each iteration shall be in a range, we {\color{black} call} it a successive
orthogonal \emph{range-relaxed} projection method.
Each iterate is obtained from the previous one by projecting it onto
a convex set that contains the solution set and in which the residual in
any point is strictly smaller that the residual at the previous iterate.
Therefore the errors as well as the residuals are strictly decreasing
along the iterates and the sequence of iterates is bounded.

\subsection{A Range-relaxed non-stationary iterated Tikhonov algorithm}
\label{ssec:true-alg}


%


%
In order to derive an implementable version of the conceptual method
(Algorithm~\ref{alg:rrNIT-concept}) discussed in the previous section, we need to specify
how to compute the range-relaxed projections (at Step [3]).
With this aim, recall that the canonical Lagrangian of problem \eqref{eq:pr-pb} is 
\begin{align}
  \label{eq:lg}
  \mathcal{L}(x,\lambda) \ = \
  \dfrac{\lambda}{2}(\norm{Ax-y^\delta}^2-\mu^2) +
  \dfrac12\norm{x-\hat x}^2 \, .
\end{align}
For each $\lambda > 0$,  $\mathcal{L}(\cdot,\lambda) : X \to \R$
has a unique minimizer $x'$ which is also characterized as
the unique solution of $\nabla_x\mathcal{L}(x,\lambda)=0$, that is,
\begin{align*}
  x' =\hat x-\lambda(I+\lambda A^*A)^{-1}A^*(A\hat x-y^\delta).
\end{align*}
%
%
The next lemma summarizes the solution theory for the projection problem
\eqref{eq:pr-pb} by means of Lagrange multiplier~\cite[Sec.~5.7]{TR96}.
Recall that ${\Proj}_\Omega$ denotes the orthogonal projection onto $\Omega$.

\begin{lemma} \label{lem:p1}
Suppose $\norm{A\hat x-y^\delta}>\mu>\delta$. The following assertions are
equivalent
\begin{enumerate}
\item\label{it:p1-1} $x'=\Proj_{\Omega_\mu}(\hat x)$;
\item\label{it:p1-2} $x'$ is \emph{the} solution of \eqref{eq:pr-pb};
\item\label{it:p1-3a} $x'=\hat x-\lambda^*(I+\lambda^* A^*A)^{-1}A^*(A\hat x-y^\delta)$,
  $\lambda^*>0$ and
  \begin{align*}
    \norm{Ax'-y^\delta}=\mu  
  \end{align*}
\end{enumerate}
\end{lemma}

\begin{proof}
  Equivalence between items \ref{it:p1-1} and \ref{it:p1-2} follows
  from the definition of orthogonal projections onto closed convex
  sets.  Equivalence between items~\ref{it:p1-2} and~\ref{it:p1-3a} is
  a classical Lagrange multipliers result (see, e.g.,
  \cite[Theorem~5.15]{TR96}).
\end{proof}

In the next lemma we address the range-relaxed projection problem~\eqref{eq:p-proj};
its proof follows from Lemma~\ref{lem:p1}.
\begin{lemma} \label{lem:p2}
Suppose $\norm{A\hat x-y^\delta}>\delta$ and $0<p<1$. The following assertions are
equivalent
\begin{enumerate}
\item\label{it:p2-1} $x'=\Proj_{\Omega_\mu}(\hat x)$ and
  $\delta\leq \mu\leq p\norm{A\hat x-y^{\delta}}+(1-p)\delta$;
\item\label{it:p2-2} $(x',\mu)\in X\times \R$ is a solution of \eqref{eq:p-proj};
\item\label{it:p2-3a} $x'=\hat x-\lambda(I+\lambda A^*A)^{-1}A^*(A\hat x-y^\delta)$,
  $\lambda>0$,
  \begin{align*}
    \delta\leq \norm{Ax'-y^\delta}
    \leq  p\norm{A\hat x - y^\delta} + (1-p)\delta,
  \end{align*}
  and $\mu=\norm{Ax' - y^\delta}$.
\end{enumerate}
\end{lemma}




It follows from Lemma~\ref{lem:p2}
that 
solving the range-relaxed projection problem in \eqref{eq:p-proj} boils down to solving the inequalities
\begin{align}
  \label{eq:rg}
  \begin{aligned}
    &
    \delta\leq\norm{Ax'-y^\delta} \leq 
  p\norm{A\hat x-y^\delta}+(1-p)\delta,
  \\ &\qquad\qquad\text{where }
  x' =\hat x-\lambda(I+\lambda A^*A)^{-1}A^*(A\hat x-y^\delta)
\end{aligned}
\end{align}
and defining
$x=x'$ and $\mu= \norm{Ax'-y^\delta}$.
We use this result to propose an implementable version of
Algorithm~\ref{alg:rrNIT-concept} as follows:

\begin{algorithm}[h!]
\begin{center}
\fbox{\parbox{14.3cm}{
$[1]$ choose an initial guess \ $x_0 \in X$;
\smallskip

$[2]$ choose \ $p \in (0,1)$, \ $\tau > 1$ and set \ $k := 0$;
\smallskip

$[3]$ while \ $\big( \norm{Ax_k^\delta-y^\delta}\, > \, \tau\delta \big)$ \ do
\smallskip

\ \ \ \ $[3.1]$ $k := k + 1$;
\smallskip

\ \ \ \ $[3.2]$ compute\, $\lambda_k$ and $x^\delta_k$ such that

\smallskip

\ \ \ \ \ \ \ \ \ \ \,$x_k^\delta \ = \ x_{k-1}^\delta - \lambda_k \,
        (I + \lambda_k A^*A)^{-1} \, A^*(Ax_{k-1}^\delta - y^\delta)$,\;\;

        \smallskip 


        \ \ \ \ \ \ \ \ \ \ \,$\delta\leq \norm{Ax_k^\delta-y^\delta}
        \leq p\norm{Ax_{k-1}^\delta - y^\delta} + (1-p) \delta$
} }
\end{center} \vskip-0.5cm
\caption{The rrNIT method.} \label{alg:rrNIT-true}
\end{algorithm}

\noindent
As in Algorithm~\ref{alg:rrNIT-concept}, the stopping index in the above algorithm
is defined by the discrepancy principle
\begin{equation} \label{def:k-star}
k^* \ := \ \min \{ k \geq 1 \, ; \ \norm{A x_j - y^\delta} > \tau\delta , \
                   j = 0, \dots, k-1 \ \ \ {\rm and} \ \ \
                   \norm{A x_{k} - y^\delta} \leq \tau\delta \} \, .
\end{equation}
The computational burden of the above algorithm resides in the
computation of step
[3.2], which requires the solution of a linear system whose
corresponding residual shall be in a given range.  In other words,
$\lambda_k$ shall be a solution of \eqref{eq:rg} with
$\hat x=x_{k-1}^\delta$.

The next lemma provides a lower bound for  $\lambda=\lambda_k$
which will be used in the convergence analysis
of Algorithm~\ref{alg:rrNIT-true}, presented in Section~\ref{sec:conv-anal}.
This lower bound can also be used for 
used as initial guess to compute $\lambda_k$.
%




\begin{lemma} \label{lm:lbound}
Under the assumptions of Lemma~\ref{lem:p1},
\begin{align*}
  \lambda^*\geq \dfrac{(\norm{A\hat x-y^\delta}-\mu)\norm{A\hat x-y^\delta}}
  {\norm{A^*(A\hat x-y^\delta)}^2}.
\end{align*}
\end{lemma}
\begin{proof}
  To simplify the notation, let
  \[
    z \ := \ \Proj_{\Omega_\mu} \hat x,\qquad b \ := \ A\hat x-y^\delta.
  \]
  From the assumption $\norm{A \hat x - y^\delta} > \mu > \delta$, it
  follows that $\hat x \notin \Omega_\delta$.
  Therefore $\norm{Az-y^\delta} = \mu$, 
  \begin{align}\label{eq:Azx}
    \norm{A(z-\hat x)}\geq\norm{A\hat x-y^\delta}-\norm{Az-y^\delta}
    =\norm{b}-\mu,
  \end{align}
  and
  \[
    \mu^2 \ = \ \norm{A(z - \hat x) + b}^2
    \ = \ \norm{A(z - \hat x)}^2 \, + \, \norm{b}^2
    \, + \, 2 \ipl A(z - \hat x) , \, b \ipr \, .
  \]
  %
  %
%
Direct combination of the above equation with the previous inequality yields
\begin{eqnarray*}
- 2 \big\ipl A(z - \hat x) , \, b \big\ipr
&   =  & \norm{A(z - \hat x)}^2 \, + \, \norm{b}^2 - \mu^2 \\
& \geq & \big( \norm{b} - \mu \big)^2 \, + \, \norm{b}^2 - \mu^2 \\
&   =  & 2 \, \norm{b} \, \big( \norm{b} - \mu \big).
\end{eqnarray*}
%
Therefore, if follows from \eqref{eq:Azx}, Cauchy-Schwartz inequality,
and the definition of $\Proj_{\Omega_\mu}$ that
\begin{eqnarray*}
\big( \norm{b} - \mu \big) \, \norm{b}
& \leq & \big\ipl -(z - \hat x) , \, A^* b \big\ipr \\
&  =   & \big\ipl \lambda^* (I + \lambda^* A^*A)^{-1} A^* b, \, A^* b \big\ipr \\
& \leq & \lambda^* \, \norm{(I + \lambda^* A^*A)^{-1}} \, \norm{A^* b}^2 \\
& \leq & \lambda^* \, \, \norm{A^* b}^2 \, ,
\end{eqnarray*}
proving the lemma.
\end{proof}

\begin{corol} \label{cor:lbd-estim}
Let the sequences $(x_k^\delta)$ and $(\lambda_k)$ be defined by the rrNIT method
(Algorithm~\ref{alg:rrNIT-true}), with $\delta \geq 0$ and $y^\delta \in Y$ as in
\eqref{eq:noisy-data} and let
$\mu_k :=\norm{Ax_k^\delta-y^\delta}$. Then
\begin{equation} \label{eq:lbd-k-estim}
\lambda_k \ \geq \
\frac{\big( \norm{A x_{k-1}^\delta - y^\delta} - \mu_k \big)
            \, \norm{A x_{k-1}^\delta - y^\delta}}
     {\norm{A^* (A x_{k-1}^\delta - y^\delta)}^2} \, , \
     k = 1, \dots, k^*.
\end{equation}
In the exact data case (i.e., $\delta = 0$) the above estimate simplifies to\,
$\lambda_k \, \geq \, (1-p) \, \norm{A}^{-2}$.
\end{corol}
\begin{proof}
From \eqref{eq:rg} and the definition of $\mu_k$ it follows that
$\delta < \mu_k < \norm{A x_{k-1}^\delta - y^\delta}$. Thus, \eqref{eq:lbd-k-estim}
follows from Lemma~\ref{lm:lbound} with $\hat x = x_{k-1}^\delta$,
$\lambda^* = \lambda_k$ and $\mu = \mu_k$ (in the proof of that lemma
it holds $z = x_k^\delta$).

In the exact data case, it follows from \eqref{eq:lbd-k-estim}, together with Assumption
\textbf{(A2)}, that $\lambda_k \, \geq \, \big( \norm{A x_{k-1} - y} - \mu_k \big)
\norm{A}^{-2} \norm{A x_{k-1} - y}^{-1}$.
Moreover, since $\delta = 0$ we have $\mu_k \leq p \norm{A x_{k-1} - y}$.
Combining these two facts, the second assertion follows.
\end{proof}

\section{Convergence Analysis} \label{sec:conv-anal}

We begin this section by establishing an estimate for the decay of the residual
$\norm{A x_k^\delta - y^\delta}$.

\begin{proposition} \label{prop:decay-residual}
Let $(x_k^\delta)$ be the sequence defined by the rrNIT method (Algorithm~\ref{alg:rrNIT-true}),
with $\delta \geq 0$ and $y^\delta \in Y$ as in \eqref{eq:noisy-data}. Then
\[
\big[ \norm{A x_k^\delta - y^\delta} - \delta \big] \ \leq \  p\,
\big[ \norm{A x_{k-1}^\delta - y^\delta} - \delta \big] \ \leq \ p^k\,
\big[ \norm{A x_0 - y^\delta} - \delta \big] \, , \ k = 1, \dots, k^* \, ,
\]
where $k^* \in \N$ is defined by \eqref{def:k-star}.
\end{proposition}
\begin{proof}
It is enough to verify the first inequality.
Recall that $x_k^\delta \in \Omega_{\mu_k}$, where $\delta \leq \mu_k \leq
p \norm{A x_{k-1}^\delta - y^\delta} + (1-p)\delta$.
Consequently,
$\norm{A x_k^\delta - y^\delta} \leq p\norm{A x_{k-1}^\delta - y^\delta} + (1-p)\delta$,
and the first inequality follows.
\end{proof}

Now we are ready to prove finiteness of $k^*$ and to provide an upper
bound for it, whenever $\delta>0$.


\begin{corol} \label{cor:estim-k*}
Let $(x_k^\delta)$ be the sequence defined by the rrNIT method (Algorithm~\ref{alg:rrNIT-true}),
with $\delta > 0$ and $y^\delta \in Y$ as in \eqref{eq:noisy-data}.
Then the stopping index $k^*$ defined in \eqref{def:k-star} satisfies
\[
k^* \ \leq \ |\ln p|^{-1} \,
             \ln \left[ \frac{\norm{A x_0 - y^\delta}-\delta}{(\tau-1)\delta} \right] + 1 \, .
\]
\end{corol}
\begin{proof}
We may assume $\norm{A x_0 - y^\delta} > \tau\delta$.%
\footnote{Otherwise the iteration does not start, i.e.,\, $k^* = 0$.}
From \eqref{def:k-star} follows $\tau\delta < \norm{A x_{k}^\delta - y^\delta}$,
$k=0, \dots\ k^*-1$. This inequality (for $k = k^*-1$), together with Proposition~%
\ref{prop:decay-residual} imply that
\[
(\tau-1)\delta \   <  \ \norm{A x_{k^*-1}^\delta - y^\delta} - \delta
               \ \leq \ p^{k^*-1} \big[ \norm{A x_0 - y^\delta} - \delta \big] \, ,
\]
completing the proof (recall that $p \in (0,1)$).
\end{proof}

Monotonicity of the iteration error $\norm{x^\star - x_k^\delta}$ was already established
in Section~\ref{sec:rrNIT}. In the next proposition we estimate the ``gain''\,
$\norm{x^\star - x_{k-1}^\delta}^2 - \norm{x^\star - x_k^\delta}^2$\, in the rrNIT method.

\begin{proposition} \label{prop:gain-estim}
Let $(x_k^\delta)$ be the sequence defined by the rrNIT method (Algorithm~\ref{alg:rrNIT-true}),
with $\delta > 0$ and $y^\delta \in Y$ as in \eqref{eq:noisy-data}. Then
\begin{equation} \label{eq:gain-estim}
\norm{x^\star - x_{k-1}^\delta}^2 - \norm{x^\star - x_k^\delta}^2
\  =  \ \norm{x_k^\delta - x_{k-1}^\delta}^2
\, + \, \lambda_k \norm{A (x^\star - x_k^\delta)}^2
\, + \, \lambda_k \, \Big[ r(x_k^\delta) - r(x^\star) \Big] \, ,
\end{equation}
for\, $k = 1, \dots, k^*$, where $r(x) := \norm{A x - y^\delta}^2$.\, Consequently,
\begin{equation} \label{eq:gain-estim-le}
\norm{x^\star - x_{k-1}^\delta}^2 - \norm{x^\star - x_k^\delta}^2
\ \geq \ \lambda_k^2 \, \norm{A^* (A x_k^\delta - y^\delta)}^2 \, + \, \lambda_k
         \Big[ \norm{A (x^\star - x_k^\delta)}^2 + (\tau^2 - 1) \, \delta^2 \Big] ,
\end{equation}
for\, $k = 1, \dots, k^*-1$; moreover,
\begin{equation} \label{eq:gain-estim-last}
\norm{x^\star - x_{k^*-1}^\delta}^2 - \norm{x^\star - x_{k^*}^\delta}^2
\ \geq \ \lambda_{k^*}^2 \, \norm{A^* (A x_{k^*}^\delta - y^\delta)}^2 \, + \,
         \lambda_{k^*} \, \norm{A (x^\star - x_{k^*}^\delta)}^2 .
\end{equation}
\end{proposition}
\begin{proof}
First we derive \eqref{eq:gain-estim}. Due to the definition of
$(x_k^\delta,\,\lambda^\delta)$ in Algorithm~\ref{alg:rrNIT-true}, the
Lagrangian $\mathcal{L}$ in \eqref{eq:lg}
(with $\mu = \mu_k$ and $\hat x = x_{k-1}$) satisfies
$$
{\cal L}(x, \lambda_k) \, = \,
{\cal L}(x_k^\delta, \lambda_k) \, + \, {\cal L}_x(x_k^\delta, \lambda_k) (x - x_k^\delta)
\, + \,
\frac 1 2 \, \big\ipl (x-x_k^\delta) ,\, H(\lambda_k) (x-x_k^\delta) \big\ipr ,
$$
where $H(\lambda_k) \,=\, (I + \lambda_k A^* A)$ is the Hessian of
${\cal L}(\cdot,\lambda_k)$ at $x=x_k^\delta$.\, Since ${\cal L}_x(x_k^\delta, \lambda_k) = 0$,
we have
\begin{align*}
  {\cal L}(x, \lambda_k)
  = 
   {\cal L}(x_k^\delta, \lambda_k) \, + \,
  \frac{1}{2}
  \big\ipl (x-x_k^\delta) , \, (I + \lambda_k A^* A) (x-x_k^\delta) \big\ipr,
\end{align*}
that is,
\begin{align*}
\norm{x-x_{k-1}^\delta}^2  +  \lambda_k \big[ r(x) - \mu_k^2 \big]
=
   \norm{x_k^\delta-x_{k-1}^\delta}^2 + \lambda_k \big[ r(x_k^\delta) - \mu_k^2 \big]
   + \norm{x - x_k^\delta}^2  +  \lambda_k \norm{A (x - x_k^\delta)}^2 .  
\end{align*}
Now, choosing $x = x^\star$, one establishes \eqref{eq:gain-estim}.
Inequality \eqref{eq:gain-estim-le}, on the other hand, follows from
\eqref{eq:gain-estim} together with $r(x^\star) \leq \delta^2$\,
and\, $r(x_k^\delta) > \tau^2\delta^2$, for\, $k = 1, \dots, k^*-1$.
Analogously, \eqref{eq:gain-estim-last} follows from \eqref{eq:gain-estim}
together with\, $r(x_{k^*}^\delta) > \delta^2$ (see Algorithm~\ref{alg:rrNIT-true}).
\end{proof}


\begin{corol} \label{cor:gain-estim-exactdata}
In the exact data case, i.e., $\delta = 0$ and $y^\delta = y \in R(A)$,
then \eqref{eq:gain-estim} becomes
\begin{equation*}
\norm{x^\star - x_{k-1}}^2 - \norm{x^\star - x_k}^2 \ = \
\norm{x_k - x_{k-1}}^2 \, + \, 2 \lambda_k \, \norm{A x_k - y}^2 ,
\end{equation*}
from which follows\,
$\sum_{k\geq 1} \, \lambda_k \, \norm{A x_k - y}^2 \, < \, \infty$ \, and \,
$\sum_{k\geq 1} \, \norm{x_k - x_{k-1}}^2 \, < \, \infty$\,
(the last inequality means that the operator describing the rrNIT iteration
is a reasonable wanderer in the sense of \cite{BrPe67}).
\end{corol}

Next we prove strong convergence of the rrNIT method (in the exact data case)
to a solution of the inverse problem \eqref{eq:ip}. The estimate in
Lemma~\ref{lm:lbound} plays a key role in this proof.


\begin{theorem} \label{theor:converg}
Let $(x_k)$ and $(\lambda_k)$ be the sequences defined by the rrNIT method
(Algorithm~\ref{alg:rrNIT-true}), with $\delta = 0$ and $y^\delta = y \in R(A)$.
Then $(x_k)$ converges strongly to some $x^* \in X$. Moreover, $A x^* = y$.
\end{theorem}
\begin{proof}
From the second assertion in Corollary~\ref{cor:lbd-estim} it follows that\,
$\sum_{k\geq 1}\, \lambda_k \, = \, \infty$.
The proof now follows from \cite[Theor.\,1.4]{BS87}.
%
\end{proof}

\section{Numerical Implementation} \label{sec:num-implementation}

In this section the implementation of Algorithm~\ref{alg:rrNIT-true} is reviewed.
We discuss the implementation of step [3.2] of that algorithm by means of a Newton-like
method, and how we accelerated this computation.

As discussed in Section~\ref{ssec:true-alg}, at step [3] of Algorithm~\ref{alg:rrNIT-true},
$\lambda_k \geq 0$ is to be obtained as a solution of the scalar rational inequalities
(or a inclusion)
\eqref{eq:rg} with $\hat x=x_{k-1}^\delta$, that is,
\begin{align}
  \label{eq:rg-k}
  \begin{aligned}
    &
    \delta\leq\norm{A_kx_k^\delta-y^\delta} \leq 
    p\norm{Ax_{k-1}^\delta-y^\delta}+(1-p)\delta,
  \\ &\qquad\qquad\text{where }
  x_k^\delta =x_{k-1}^\delta-\lambda(I+\lambda A^*A)^{-1}A^*(Ax_k^\delta-y^\delta)
  \text{ with }\lambda>0.
\end{aligned}
\end{align}
Define, at iteration $k$, $\pi_k(\lambda)$ as the candidate for
$x_k^\delta$ obtained from $x_{k-1}^\delta$ with the Lagrange
multiplier $\lambda$ and let $G_k(\lambda)$ be the square residual at
that point, that is,
\begin{subequations}
\begin{align}
\label{eq:G-pi1}
  \pi_k(\lambda) & :=
                   x_{k-1}^\delta-
                   \lambda(I+\lambda A^*A)^{-1}A^*(A x_{k-1}^\delta-y^\delta),
  \\
   & :=
     (\lambda^{-1}I+ A^*A)^{-1}(\lambda^{-1} x_{k-1}+A^*y^\delta),
  \\
  \label{eq:G-pi2}
 G_k(\lambda) & := \norm{A\pi_k(\lambda)-y^\delta}^2 .
\end{align}
\end{subequations}
With this notation, \eqref{eq:rg-k} writes
\begin{align}
  \label{eq:rat-ineq}
  \delta^2\leq  G_k(\lambda)\leq (p\norm{Ax_{k-1}^\delta-y^{\delta}}+(1-p)\delta)^2
  \quad\text{ with }\lambda>0.
\end{align}

As earlier mentioned, the set of feasible Lagrange multipliers is a
non-degenerate interval in each iteration of the rrNIT method.

\begin{proposition}
  \label{pr:ndi}
  Suppose Algorithm~\ref{alg:rrNIT-true} reaches iteration $k$ and
  $\norm{Ax_k^\delta-y^\delta} > \tau\delta$. Then, for $\lambda \geq 0$
  \begin{align*}
    \dfrac{d}{d\,\lambda}G_k(\lambda)
    &= -2 \left\langle A^*( A\pi_k(\lambda)-y^\delta) ,
      (I+\lambda A^*A)^{-1}A^*(A\pi_k(\lambda)-y^\delta) \right\rangle\\
    &= -2\lambda^{-3} \left\langle x_{k-1}-\pi_k(\lambda) ,
      (\lambda^{-1}I+A^*A)^{-1}(x_{k-1}-\pi_k(\lambda)) \right\rangle.
  \end{align*}
  Moreover, $G_k$ is strictly decreasing in $(0,\infty)$,
   the solution set of~\eqref{eq:rat-ineq} is
  $[\lambda_{\min},\lambda_{\max}]\cap\R$ where
  \begin{align*}
    \lambda_{\min}
    &:=\min\set{\lambda>0\,:\,G_k\leq
      (p\norm{Ax_{k-1}^\delta-y^{\delta}}+(1-p)\delta)^2}\\
    \lambda_{\max}&:=\sup\set{\lambda>0\,:\,\delta^2\leq G_k(\lambda)},
  \end{align*}
  and $0<\lambda_{\min}<\lambda_{\max}\leq\infty$.
\end{proposition}

We will solve \eqref{eq:rat-ineq} by means of a Newton-type method.
%
Newton's method for solving \eqref{eq:rat-ineq},
would be to take some $\lambda_{k,0}>$ and to iterate
\begin{align*}
  \lambda_{k,j+1} \ = \ \lambda_{k,j}-\dfrac{G_k(\lambda_{k,j})-\delta^2}
                        {G'_k(\lambda_{k,j})}
\end{align*}
as long as \eqref{eq:rat-ineq} is not satisfied; when \eqref{eq:rat-ineq}
is satisfied, the last $\lambda_{k,j}$ is used as $\lambda_k$. 
We will introduce a number of modifications in this iteration to accelerate it:
\begin{enumerate}
\item[(M1)] A ``greedy'' version of Newton's method will be used, aiming at
  $G_k(\lambda)=0$, that is, the numerator on the above fraction
  will be $G_k(\lambda_{k,j})$.
\item[(M2)] Newton's step will be dynamically over-relaxed by a factor $w_j$,
  as described below.
\item[(M3)] We choose $\lambda_{k,0}$ using Lemma~\ref{lm:lbound} for $k=0$ and
  information gathered at previous iterations of $k\geq 2$, ad
  described below.
\end{enumerate}

\noindent  Regarding modifications (M1) and (M2), while $\lambda=\lambda_{k,j}$ does not
satisfy~\eqref{eq:rat-ineq}, we use the iteration
\begin{equation} \label{eq:newton-step-or}
\lambda_{k,j+1} \ := \ \lambda_{k,j}-\omega_j\dfrac{G_k(\lambda_{k,j})}
                       {G'_k(\lambda_{k,j})} \, ,
\end{equation}
where  the over-relaxation factor $\omega_j$ is chosen as follows: \\
--- for $j=0$, $\omega_0 = 1$; \\
--- for $j \geq 1$, after computing $\lambda_{k,j}$,  $\pi_k(\lambda_{k,j})$,
and
    $G_k(\lambda_{k,j})$,\\
\mbox{\ \ \ \ \ \ \ \ }
{\bf if} \ \,$G_k(\lambda_{k,j-1})>2 \big( p\norm{A x_{k-1}^\delta - y^\delta} + (1-p)\delta \big)^2$ \\
\mbox{\ \ \ \ \ \ \ \ \ \ \ \ }
     {\bf then} \  $\omega_{j} \, = \, 2\, \omega_{j-1}$ \\
     \mbox{\ \ \ \ \ \ \ \ \ \ \ \ } {\bf else} \ \
     $\omega_{j} \, = \, 1$

\noindent  Regarding modification (M3):
     \\
     --- for $k=1$, $\lambda_{1,0}$ is the lower bound provided by Lemma~\ref{lm:lbound}; \\
     --- for $k =2$, $\lambda_{2,0}=\lambda_1$,
     \\
     --- for $k \geq 3$,
we use a linear extrapolation on $\log$ $\lambda$ from the two previous iterates
as starting point, that is,  $\lambda_{k,0} = {\lambda_{k-1}^2} / {\lambda_{k-2}}$.

The acceleration effect caused by modifications (M1), (M2) and (M3)
is illustrated in Example~\ref{ex:acceleration} below.

Over-relaxation is a well established technique for accelerating iterative
methods for solving linear and non-linear equations, the SOR method being
a classical example.
This {\color{black} fact} motivated the introduction of over-relaxation as in (M2).

In our numerical experiments, we observed that the sequence $\lambda_k$
increases exponentially. This fact motivated the use of (liner) extrapolation
(in the $\log$) for its initial value from iteration 3 on (modification (M3)).

It is worth noticing that, in step 3, each ``inner iteration'' \eqref{eq:newton-step-or}
requires the solution of either two linear systems per failed inner iteration or one
linear system at the successful last inner iteration.
On the other hand, in the gNIT method the computation of $\lambda_k = q^k$
(for some {\em a priori} chosen $q > 1$) is straightforward. Consequently,
one needs to solve one linear system (modeled by $(I + q^k A^*A)$) in each
step of the gNIT method.
This facts motivated us to use the accumulated number of linear system to measure
the performance of the different NIT method by plotting the residual and the error 
as a function of this quantity (see Figure~\ref{fig:rrNIT-acceleration1}).

\begin{example} \label{ex:acceleration}
The benchmark problem presented in Example~\ref{ex:gNIT-sIT} is revisited and
solved by the rrNIT method using the Newton-method \eqref{eq:newton-step-or}
with combinations of modifications (M1)-(M3):
In Figure~\ref{fig:rrNIT-acceleration1} this inverse problem is solved using:
\\ (BLUE) 
modification (M1);
\\ (PINK) 
modifications (M1), (M2);
\\ (RED) 
modifications (M1), (M2), (M3);
\\
Notice that, in Figure~\ref{fig:rrNIT-acceleration1}, the x-axis denotes the
accumulated number of linear systems. This choice allows an better comparison
of the efficiency of the different rrNIT implementations.
\end{example}

\begin{figure}[t]
\centerline{\includegraphics[width=\textwidth]{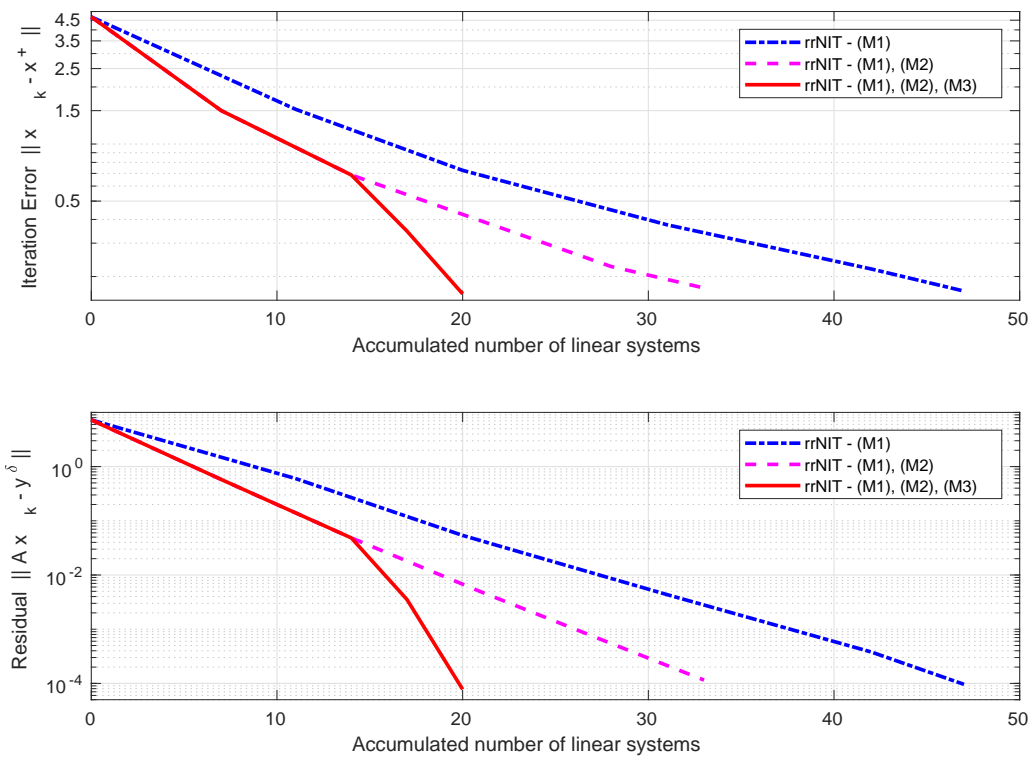}}
\vskip-0.5cm
\caption{\small Implementation of rrNIT method: acceleration caused
by modifications (M1), (M2), (M3) to iteration \eqref{eq:newton-step-or}.}
\label{fig:rrNIT-acceleration1}
\end{figure}

Figure~\ref{fig:rrNIT-acceleration1} illustrates that the \emph{cumulative}
effect of modifications (M1), (M2), and (M3) is to accelerate the computation
of $\lambda_k$ as required in step [3] of Algorithm~\ref{alg:rrNIT-true}.

A pseudo-code version of our implementation of Algorithm~\ref{alg:rrNIT-true}
with the strategies above discussed is presented in Appendix~A for
the sake of completeness.

\section{Numerical experiments} \label{sec:num-experiments}

In this section, Algorithm~\ref{alg:rrNIT-numer} (see Appendix~A)
{\color{black} is implemented} for solving two well known linear ill-posed problems.
In Section~\ref{ssec:num-id} the {\em Image Deblurring Problem}
\cite{BeBo98, Ber09} is considered, while in Section~\ref{ssec:num-ipp}
we address the {\em Inverse Potential Problem} \cite{HeRu96, Isa06} in 2D,
which is an elliptic parameter identification problem.

In both Sections~\ref{ssec:num-id} and~\ref{ssec:num-ipp}, the performance
of Algorithm~\ref{alg:rrNIT-numer} is compared to the performance of two well
established methods: The gNIT method with $\lambda_k = 2^k$; the NIT method
proposed in \cite{DoHa13}.

\subsection{Image deblurring problem} \label{ssec:num-id}

Image deblurring problems \cite{BeBo98}
are finite dimensional problems modeled, in general, by high dimensional linear
systems of the form \eqref{eq:ip}.
In this setting, 
$x \in X = \R^n$ represents the pixel values of an unknown true image, while
$y \in Y = X$ contains the pixel values of the observed (blurred) image.
In practice, only noisy blurred data $y^\delta \in Y$ satisfying \eqref{eq:noisy-data}
is available.

The matrix $A$ describes the blurring phenomenon \cite{Ber09, BeBo98}. We consider the simple
situation where the blur of the image is modeled by a space invariant point spread function
(PSF). In the continuous model, the blurring process is represented by an integral
operator of convolution type and \eqref{eq:ip} corresponds to an integral equation of
the first kind \cite{EngHanNeu96}.
In our discrete setting, after incorporating appropriate boundary conditions into the model,
the discrete convolution is evaluated by means of the FFT algorithm.
We added to the exact data (the convoluted image) a normally distributed noise with zero mean
and suitable variance for achieving a prescribed relative noise level .

The computation of our deblurring experiment was conducted using MATLAB~2012a.
The corresponding setup is shown in Figure~\ref{fig:ID-setup}:
(a) True image $x \in \R^n$, $n = 256^2$ (Cameraman $256\times256$);
(b) PSF is the rotationally symmetric Gaussian low-pass filter of size [257 257] and
standard deviation $\sigma = 4$ (command {\tt fspecial('gaussian',\,[257 257],\,4.0)});
(c) Exact data $y = Ax \in \R^n$ (blurred image).
The noise was generated used the {\tt randn} routine while the FFT was computed
using the {\tt fft2} routine.

\begin{figure}[t]
\centerline{\includegraphics[width=0.35\textwidth]{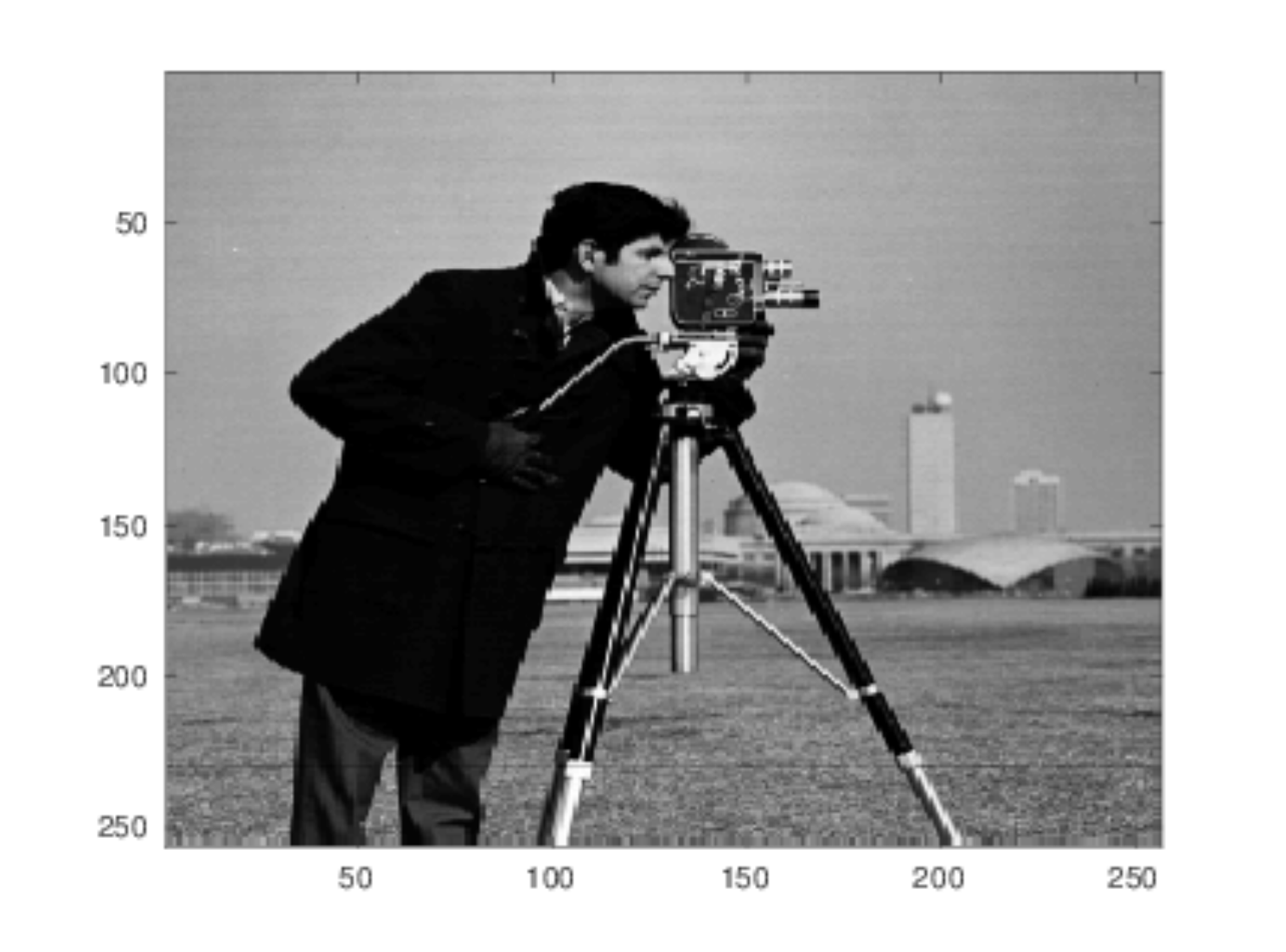}
            \includegraphics[width=0.37\textwidth]{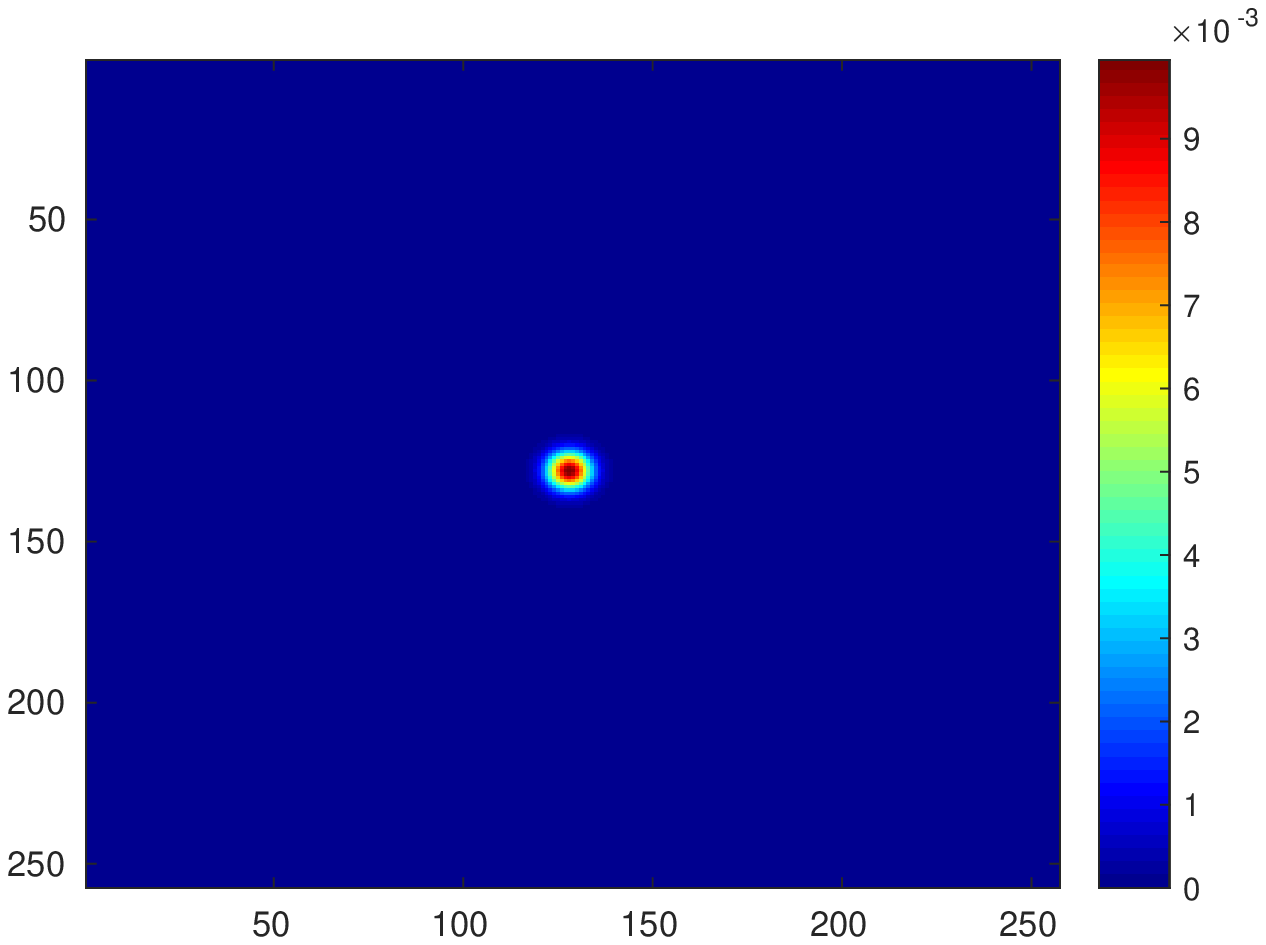}
            \includegraphics[width=0.35\textwidth]{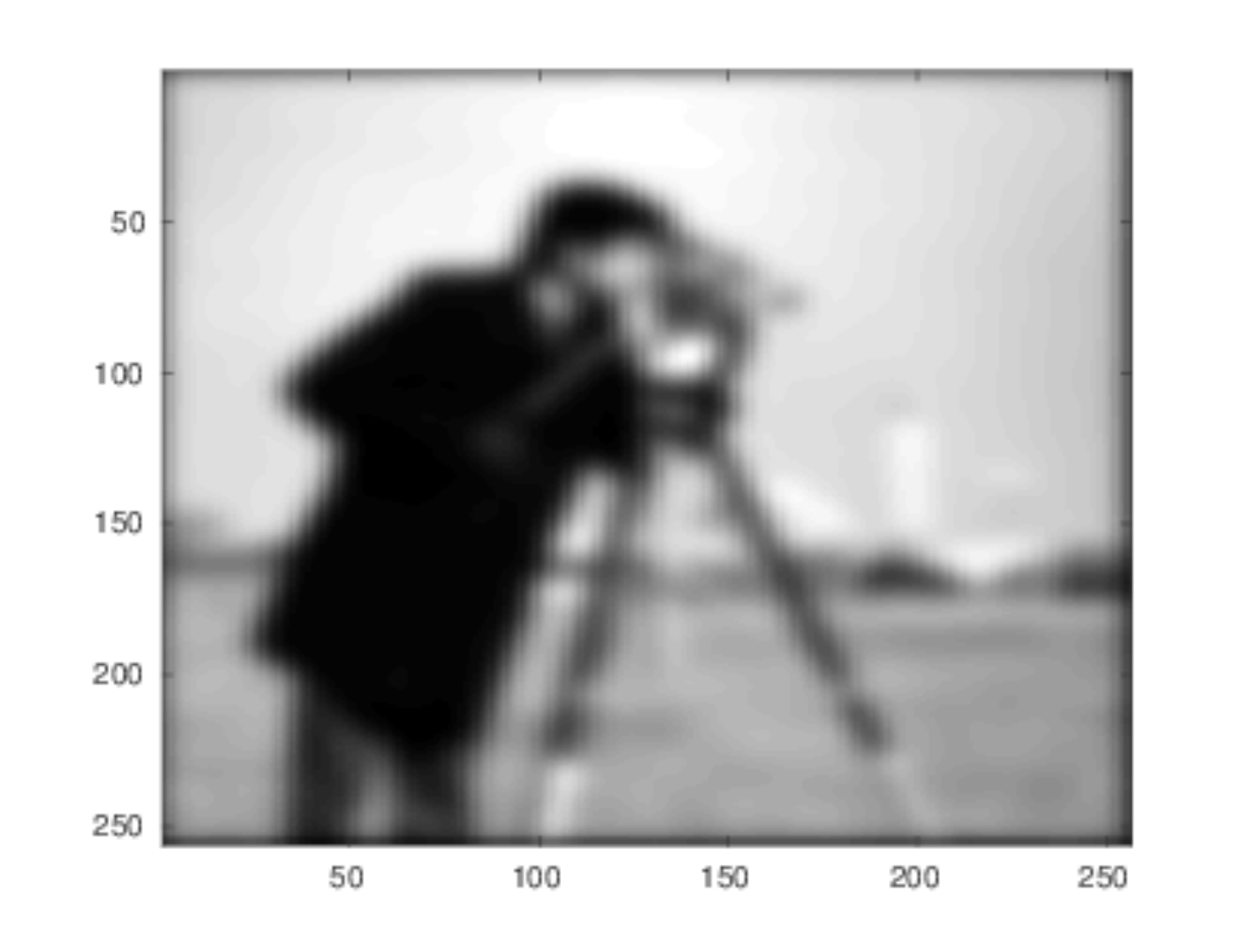} } \vskip0.3cm
\vskip-0.5cm \centerline{\hfill (a) \hspace{5cm} (b) \hspace{5cm} (c) \hfill}
\caption{\small Image deblurring problem: setup of the inverse problem. \
(a) Original image $x$; \ (b) Point spread function; \ (c) Blurred image $y$.}
\label{fig:ID-setup}
\end{figure}

Three distinct scenarios are considered, where the relative noise level
$\norm{y-y^\delta} / \norm{y}$ corresponds to $10^{-1} \%$, $10^{-3} \%$ and
$10^{-6}\%$ respectively (in the third scenario, the choice of the noise level
is motivated by MATLAB's double-precision accuracy).

In Figure~\ref{fig:ID-numerics} the following methods are compared for the third
scenario:
(BLACK) gNIT with $\lambda_k = 2^k$;
(RED) rrNIT method in Algorithm~\ref{alg:rrNIT-numer} (with $p = 0.2$);
(BLUE) Hanke-Donatelli NIT method in \cite{DoHa13}.
The pictures in Figure~\ref{fig:ID-numerics} show:
(TOP) relative error $\norm{x^\star - x_k^\delta} / \norm{x^\star}$;
(BOTTOM) residual $\norm{A x_k^\delta - y^\delta}$.
The x-axis in the these pictures is scaled by the accumulated number of linear systems solved.
This choice allows an easier comparison between the efficiency of the different methods.

All methods are stopped according to the discrepancy principle with $\tau = 3$.
As initial guess we choose $x_0 = y^\delta$ (the noisy blurred image).

The numerical results concerning all three scenarios are summarized in Table~\ref{tab:id}.
In this table we show, for each scenario, the total number of linear systems solved, as
well as the number of iterations needed to reach the stop criteria.

\begin{table}[h]
\begin{center}
\begin{tabular}{c  c c c}
\hline
     {$\delta$}    &    gNIT   & NIT in \cite{DoHa13} &   rrNIT   \\
\hline
$10^{-1} \%$ \ \ \ & \ 6 (\ 6) &        15 (\ 5)      & \ 7 (\ 4) \\
$10^{-3} \%$ \ \ \ &  17 (17)  &        23 (\ 7)      &  11 (\ 7) \\
$10^{-6} \%$ \ \ \ &  36 (36)  &        43 (11)       &  16 (11)  \\
\end{tabular}
\end{center}
\caption{Image deblurring problem: total number of linear system solves
for different noise levels with the number of iterations in parentheses.}
\label{tab:id}
\end{table}

\begin{figure}[t]
\centerline{\includegraphics[width=\textwidth]{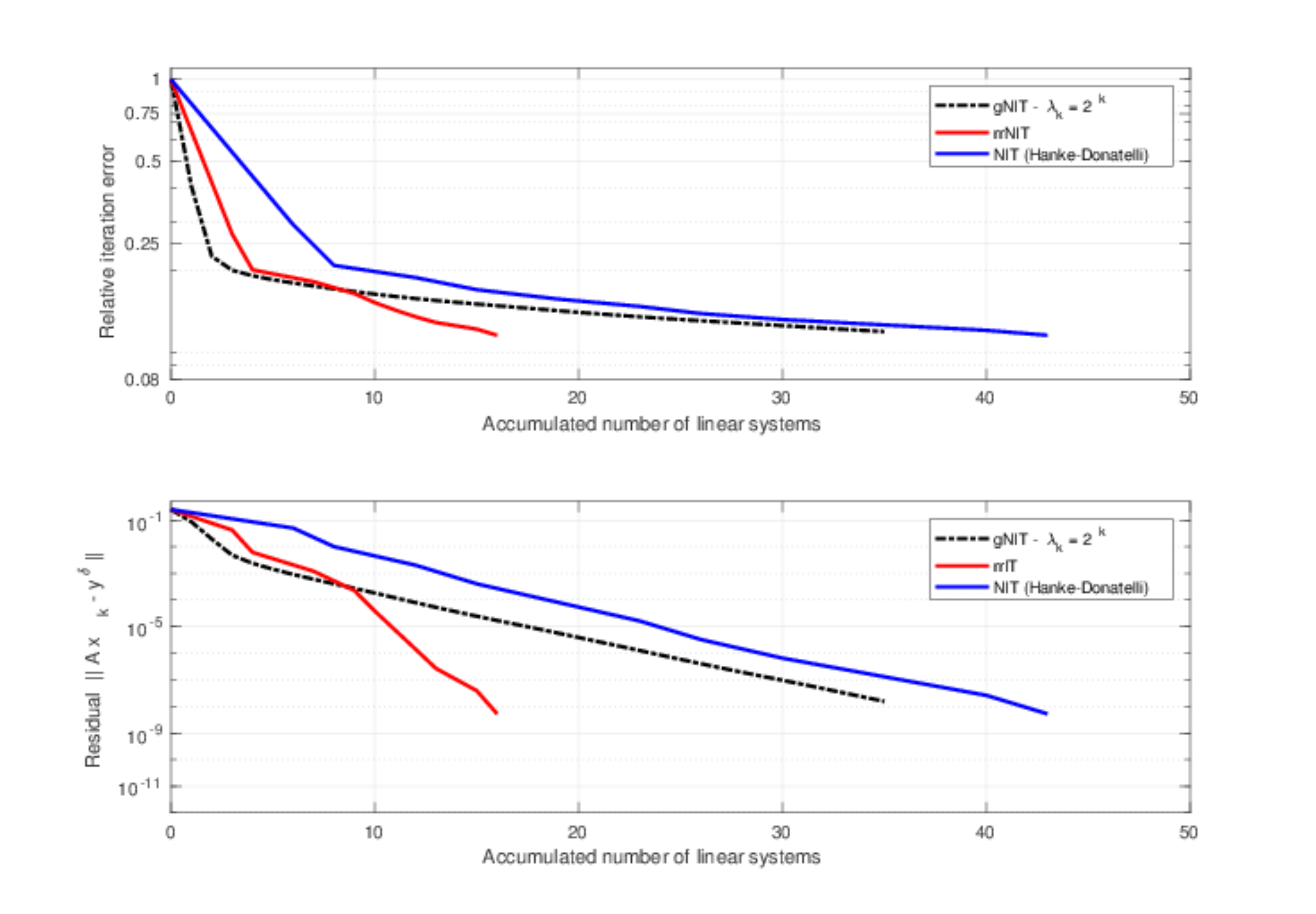} }
\vskip-0.5cm
\caption{\small Image deblurring problem: {\color{black} third scenario $\delta = 10^{-6} \%$}.
(TOP) Relative iteration error; (BOTTOM) Residual.}
\label{fig:ID-numerics}
\end{figure}

The restored images for the third scenario ($\delta = 10^{-6} \%$) are presented
in Figure~\ref{fig:ID-restored}. From left to right: gNIT, NIT in \cite{DoHa13},
and rrNIT.
\begin{figure}[t]
\centerline{\includegraphics[width=0.35\textwidth]{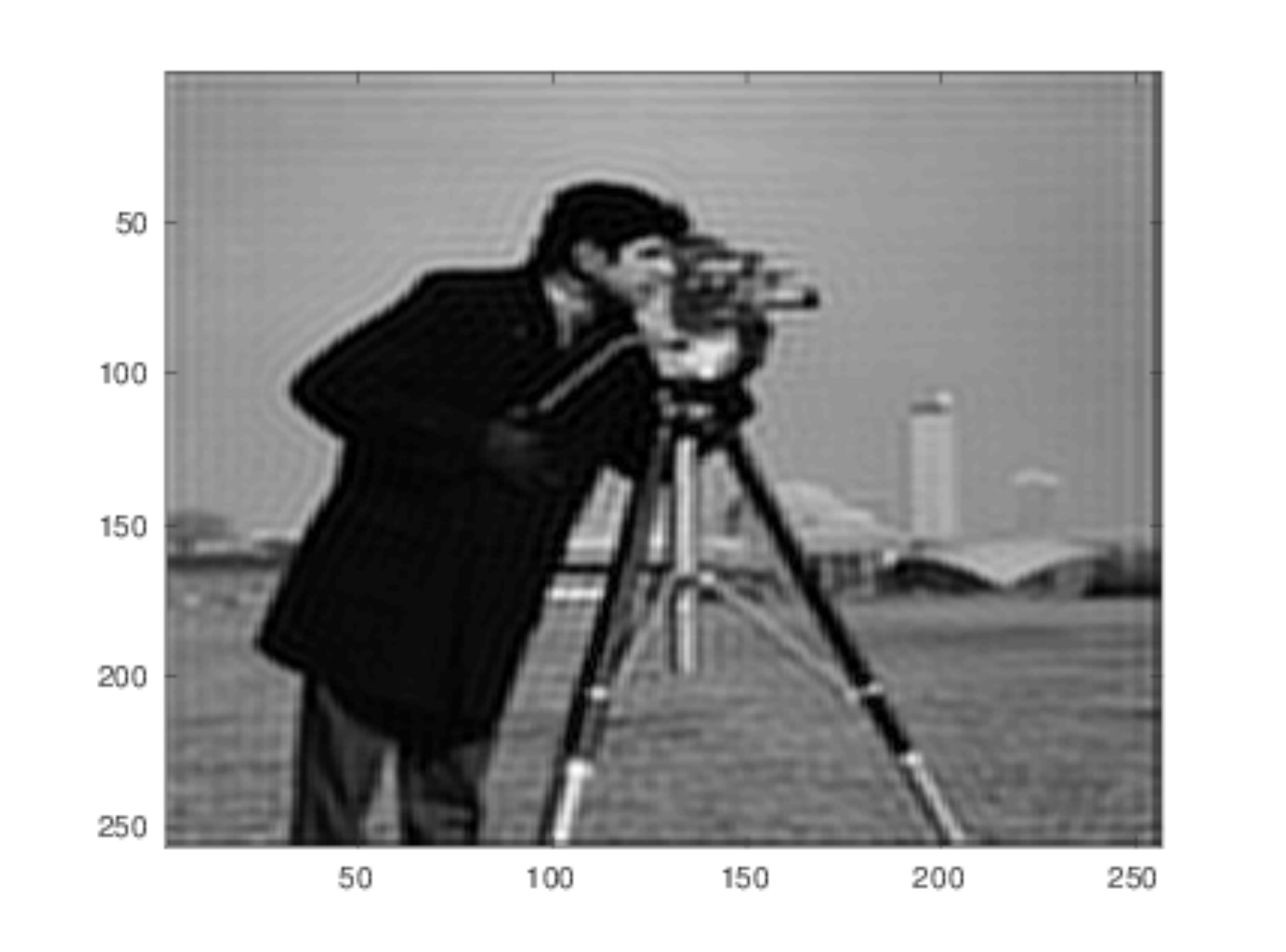}
            \includegraphics[width=0.35\textwidth]{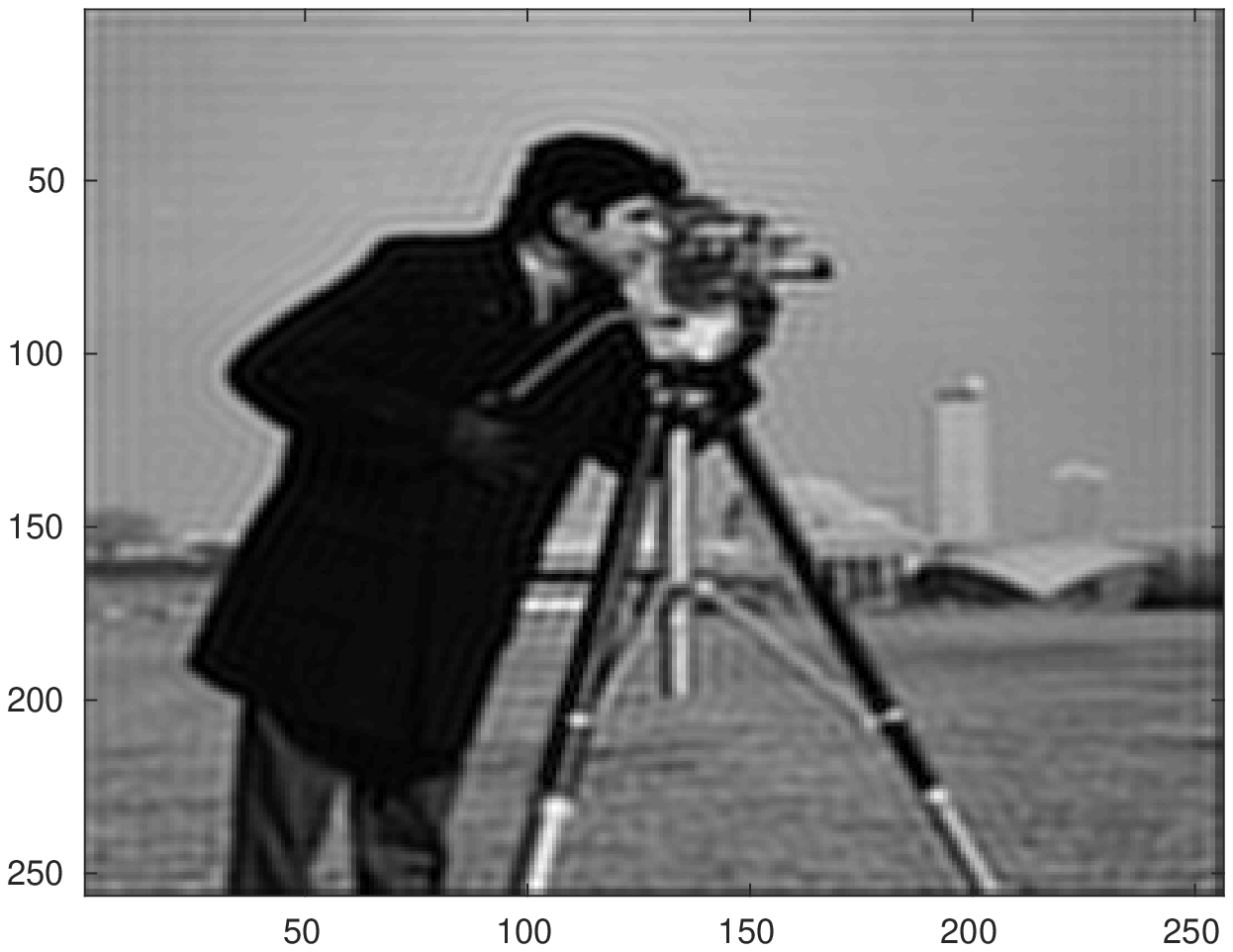}
            \includegraphics[width=0.35\textwidth]{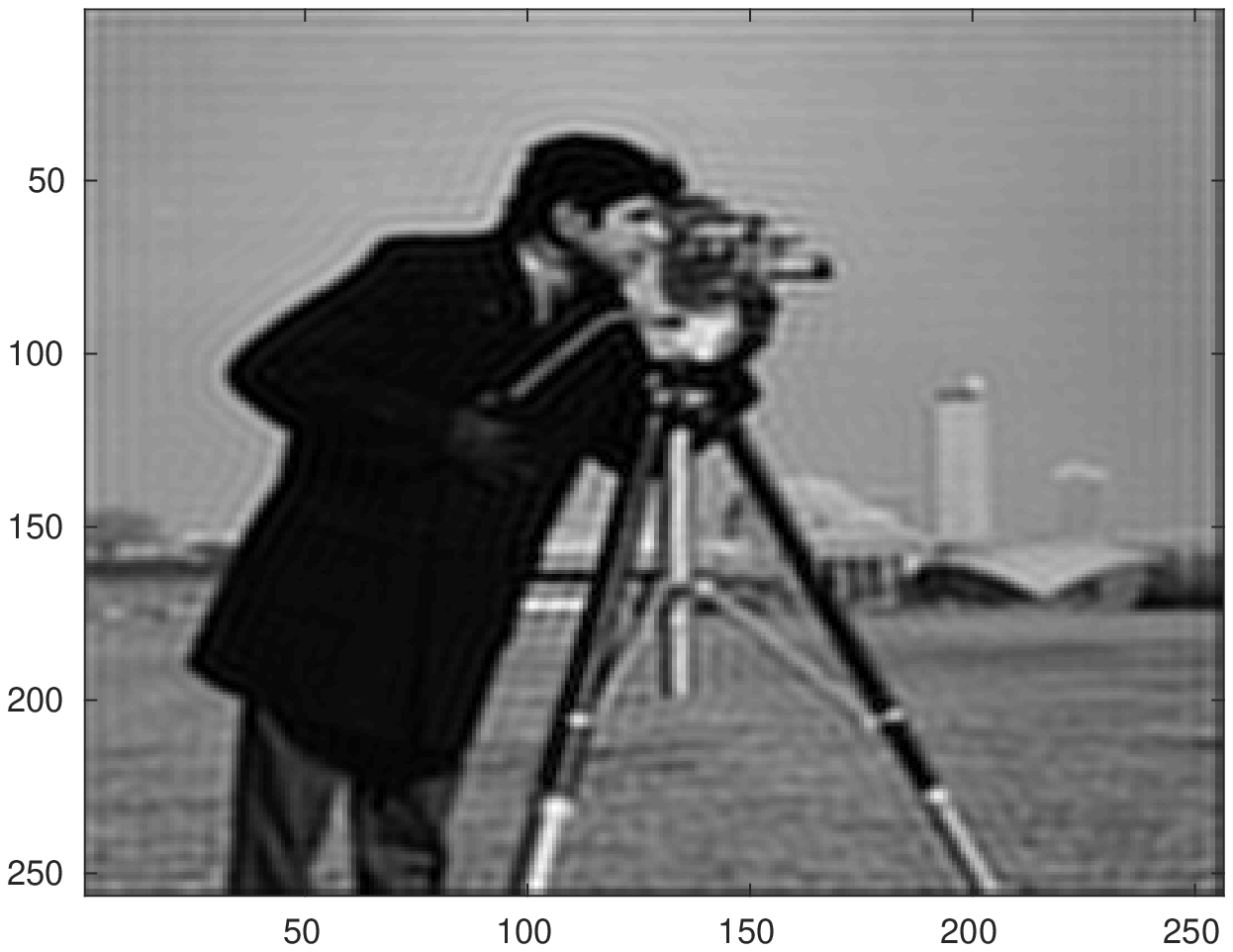} } \vskip0.3cm
\vskip-0.5cm \centerline{\hfill (a) \hspace{5cm} (b) \hspace{5cm} (c) \hfill}
\caption{\small Image deblurring problem: restored images for the third scenario
($\delta = 10^{-6} \%$). \ (a) gNIT, \ (b) NIT in \cite{DoHa13}, \ (c) rrNIT.}
\label{fig:ID-restored}
\end{figure}

\subsection{Inverse potential problem} \label{ssec:num-ipp}

In what follows we address the {\em inverse potential problem}
\cite{FSL05, CLT09, HeRu96, DAL09}.
Generalizations of this inverse problem appear in many relevant applications
including Inverse Gravimetry \cite{Isa06, DAL09}, EEG \cite{EF06}, and EMG \cite{DAP08}.

The forward problem considered here consists in solving on a Lipschitz domain
$\Omega \subset \R^d$, for a given source function $x \in L_2(\Omega)$, the
boundary value problem
\begin{equation} \label{eq:ipp}
-\Delta u \ = \ x \, ,\ {\rm in} \ \Omega \, , \quad
u \ = \ 0 \, \ {\rm on} \ \partial\Omega\, .
\end{equation}
The corresponding inverse problem is the so called {\em inverse potential problem}
(IPP), which consists of recovering an $L_2$--function $x$, from measurements of
the Dirichlet data of its corresponding potential on the boundary of $\Omega$, i.e.,
$y := u_{\nu} |_{\partial\Omega} \in L_2(\partial\Omega)$.
This problem is modeled by the linear operator 
$A: L_2(\Omega) \to L_2(\partial\Omega)$ defined by
$A x := u_{\nu} |_{\partial\Omega}$, where $u \in H_0^1(\Omega)$ is the
unique solution of \eqref{eq:ipp} \cite{HeRu96}.  Using this notation,
the IPP can be written in the abbreviated form \eqref{eq:ip}, where the
available noisy data $y^\delta \in L_2(\partial\Omega)$ satisfies
\eqref{eq:noisy-data}.

In our experiments we follow \cite{CLT09} in the experimental setup, selecting
$\Omega = (0,1) \times (0,1)$ and assuming that the unknown parameter $x^\star$
is an $H^1$-function with sharp gradients shown in Figure~\ref{fig:IPP-setup}~%
(a).
After solving Problem~\eqref{eq:ipp} for such $x= x^\star$, we added to the exact
Dirichlet data a normally distributed noise with zero mean and suitable variance
for achieving a prescribed relative noise level.
In our numerical implementations we set $p = 0.1$, $\tau = 3$ (discrepancy
principle constant) and the initial guess $x_0 \equiv 1.5$ (constant function in $\Omega$).

As in Section~\ref{ssec:num-id}, three distinct scenarios are considered, where
the relative noise level $\norm{y-y^\delta} / \norm{y}$ corresponds to
$10^{-1} \%$, $10^{-3} \%$ and $10^{-6}\%$ respectively.

In Figure~\ref{fig:IPP-numerics} the following methods are compared for the second
scenario:
(BLACK) gNIT with $\lambda_k = 2^k$;
(RED) rrNIT method in Algorithm~\ref{alg:rrNIT-numer} (with $p = 0.1$);
(BLUE) NIT method proposed in \cite{DoHa13}.%
\footnote{As before, this NIT method was implemented with $q = 0.6$ and $\rho = 10^{-4}$.
For the computation of the Lagrange multipliers, a scalar equation was solved using an
over-relaxed Newton method and a precision of $5 \%$.}
What concerns Algorithm~\ref{alg:rrNIT-numer}, the corresponding iterate $x_5^\delta$ and
iteration error $|x^\star - x_5^\delta|$ are shown in Figure~\ref{fig:IPP-setup}~(b)
and~(c) respectively.

The pictures in Figure~\ref{fig:IPP-numerics} show:
(TOP) relative error $\norm{x^\star - x_k^\delta} / \norm{x^\star}$;
(BOTTOM) residual $\norm{A x_k^\delta - y^\delta}$.
The x-axis in the these pictures is scaled by the accumulated number of linear systems solved.

The numerical results concerning all three scenarios are summarized in Table~\ref{tab:ipp}.
In this table we show, for each scenario, the total number of linear systems solved, as
well as the number of iterations needed to reach the stop criteria.

\begin{table}[h]
\begin{center}
\begin{tabular}{r  c c c}
\hline
          {}       &    gNIT   & NIT in \cite{DoHa13} &   rrNIT   \\
\hline
$10^{-1} \%$ \ \ \ & \ 6 (\ 6) &        11 (\ 3)      &   6 (\ 3) \\
$10^{-3} \%$ \ \ \ &  10 (10)  &        34 (\ 5)      &  10 (\ 5) \\
$10^{-6} \%$ \ \ \ &  13 (13)  &        86 (\ 7)      &  12 (\ 6)  \\
\end{tabular}
\end{center}
\caption{Inverse potential problem: total number of linear systems for
different noise levels with the number of iterations in parentheses.}
\label{tab:ipp}
\end{table}

\subsection{Remarks} \label{ssec:num-remarks}

In the two above discussed inverse problems, for all scenarios, both {\em a posteriori}
NIT type methods (rrNIT in Algorithm~\ref{alg:rrNIT-numer} and NIT in \cite{DoHa13})
require similar number of steps to reach discrepancy.
However, the total numerical effort of rrNIT is much smaller than the one of
NIT in \cite{DoHa13}, and is comparable to the total numerical effort of the
gNIT method (see Tables\ref{tab:id} and~\ref{tab:ipp}).

Specially in the third scenario (small noise level), the NIT method \cite{DoHa13}
needs several Newton steps to compute the Lagrange multipliers in the final iterations.
Algorithm~\ref{alg:rrNIT-numer}, on the other hand, needs only {\color{black} a few}
Newton steps to compute each one of the Lagrange multipliers solving \eqref{eq:rat-ineq}.
\smallskip

Notice the exponential decay of the residual in the rrNIT method (Figures~\ref{fig:ID-numerics}
and~\ref{fig:IPP-numerics}), which is in accordance to Proposition~\ref{prop:decay-residual}.
We also observed exponential growth of the corresponding Lagrange multipliers.
\smallskip

The NIT method proposed in \cite{DoHa13} was implemented with $q = 0.6$ and
$\rho = 10^{-4}$ as described in \cite[Sec.5]{DoHa13}. For the computation of
the Lagrange multipliers, a scalar equation was solved using an over-relaxed Newton
method with a precision of $1\%$ (Section~\ref{ssec:num-id}) and $5 \%$
(Section~\ref{ssec:num-ipp}).

\begin{figure}[t]
\centerline{\includegraphics[width=1.25\textwidth]{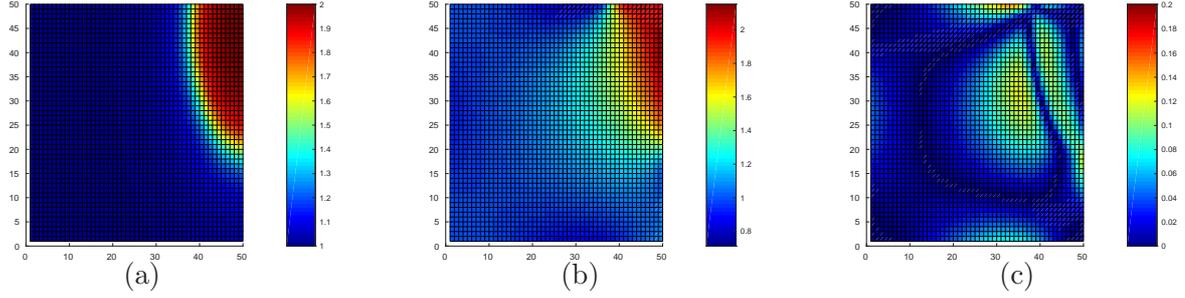} }
\vskip-0.4cm \centerline{\hfil (a) \hspace{5.0cm} (b) \hspace{5.0cm} (c) \hfil}
\vskip-0.2cm
\caption{\small Inverse potential problem: second scenario $\delta = 10^{-3} \%$.
\ (a) Exact solution $x^\star$;
\ (b) Approximate solution $x_5^\delta$ (rrNIT Algorithm~\ref{alg:rrNIT-numer});
\ (c) Iteration error $|x^\star - x_5^\delta|$ {\color{black} (absolute error pixel-wise)}.}
\label{fig:IPP-setup}
\end{figure}

\begin{figure}[ht]
\centerline{\includegraphics[width=\textwidth]{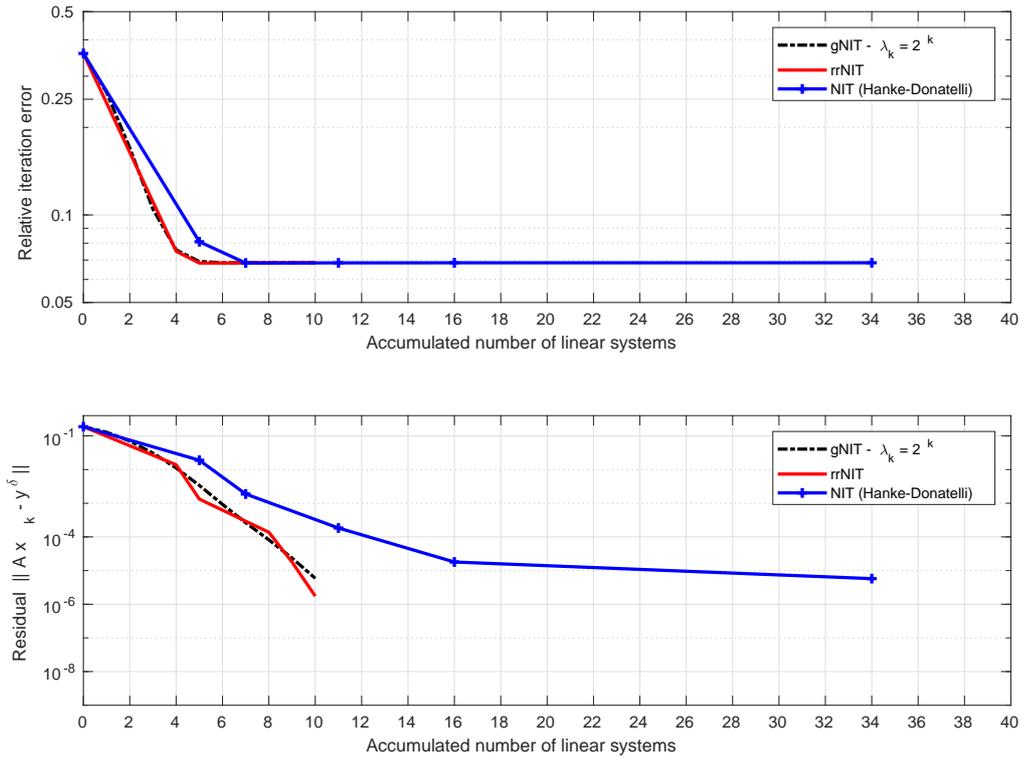} }
\vskip-0.5cm
\caption{\small Inverse potential problem: second scenario $\delta = 10^{-3} \%$.}
\label{fig:IPP-numerics}
\end{figure}

\section{Conclusions} \label{sec:conclusion}

We investigate NIT type methods for computing stable approximate
solutions to ill-posed linear operator equations.  The main
contributions of this article is a novel strategy for choosing a
sequence of Lagrange multipliers for the NIT iteration, allowing each
of this multipliers to belong to a non-degenerate interval.
We also
derived an efficient numerical algorithm based on this strategy
for computing the required Lagrange multipliers.

We prove monotonicity of the proposed rrNIT method, and exponential
decay of the residual $\norm{A x_k^\delta - y^\delta}^2$.
Moreover, we provide estimates to the``gain''\,
$\norm{x^\star - x_{k-1}^\delta}^2 - \norm{x^\star - x_k^\delta}^2$,
and to the Lagrange multipliers $\lambda_k$. A convergence proof
in the case of exact data is provided.

An algorithmic implementation of the rrNIT method is proposed
(Algorithm~\ref{alg:rrNIT-numer}, where the computation of Lagrange
multipliers are accomplished using an over relaxed Newton-like method,
with appropriate choice of the initial guess.
The resulting rrNIT method is competitive with gNIT and also with other
commonly used \emph{a posteriori} method ; not only from the point of
view of the total number of iterations, but also from the point of
view of the overall numerical effort required.

Our algorithm is tested for two well known applications with three noise
levels: the inverse potential problem, and the image deblurring
problem. The results obtained validate the efficiency of our method.

\section*{Acknowledgments}

The work of R.B. is supported by the research council of  the Alpen-Adria-Universit\"at
Klagenfurt (AAU) and by the Karl Popper Kolleg ``Modeling-Simulation-Optimization''
funded by the AAU and by the Carinthian Economic Promotion Fund (KWF).
A.L. acknowledges support from the research agencies CAPES, CNPq
(grant 311087/2017-5), and from the AvH Foundation.
The work of B.F.S. was partially supported by CNPq (grants 474996/2013-1,
302962/2011-5) and FAPERJ (grant E-26/102.940/2011).

We thank the anonymous referees for the constructive criticism and corrections
which improved the original version of this work.

\appendix
\section*{Appendix A} \label{ap:pcode}

In what follows we present a detailed algorithm for the rrNIT method, which
takes into account the above discussed strategies, namely: initial guess
choice and  over-relaxation.

Algorithm~\ref{alg:rrNIT-numer}, an implementable rrNIT method for solving
ill-posed linear problems, is written in a tutorial way.%
\footnote{Indeed, the inversion of $(I + \lambda A^*A)$ is not always possible.}
Presented in this form, one recognizes that the major computational effort in each
iteration consists in the computation of the $M_\lambda$ operators.
In the first iteration ($k=1$) this task is solved in steps [3.3] and [3.7]; in the
subsequent iterations it is solved in the Newton-method [3.7].

The above discussed choice of the initial guess $\lambda_{k,0}$ for the Newton-method
\eqref{eq:newton-step-or} is evaluated in step [3.3]. Moreover, the computation of
the over-relaxation parameters $\omega_j$ is implemented in loop [3.7].
\begin{algorithm}
\begin{center}
\fbox{\parbox{14.2cm}{
$[1]$ choose an initial guess \ $x_0 \in X$;
      \vspace{0.1cm}

$[2]$ choose \ $p \in (0,1)$, \ $\tau > 1$ \ and set \ $k := 0$;
      \vspace{0.1cm}

$[3]$ while \ $\big( \norm{Ax_k - y^\delta}_Y \, > \, \tau\delta \big)$ \ do
      \vspace{0.1cm}

\ \ \ \ $[3.1]$ $k \, := \, k + 1$;
        \vspace{0.1cm}

\ \ \ \ $[3.2]$ $\theta_k \, := \, p\norm{Ax_{k-1}^\delta - y^\delta}_Y + (1-p)\delta$;
        \vspace{0.1cm}

\ \ \ \ $[3.3]$ if\, ($k = 1$)\, then \\
\mbox{\hskip1.8cm}  $\lambda_{k,0} \, := \, \norm{A x_{k-1}^\delta - y^\delta}
                    \big( \norm{A x_{k-1}^\delta - y_\delta} - \theta_k \big)
                    \, / \, \norm{A^* (A x_{k-1}^\delta - y^\delta)}^2$; \\
\mbox{\hskip1.8cm}  $M_{\lambda_{k,0}} \, := \, (I + \lambda_{k,0} A^*A)^{-1}$; \\
\mbox{\ \ \ \ \ \ \ \ \ \,} else \\
\mbox{\hskip1.8cm}  $\lambda_{k,0} \, := \, \lambda_{k-1}$; \\
\mbox{\hskip1.8cm}  $M_{\lambda_{k,0}} \, := \, M_{\lambda_{k-1,1}}$; \\
\mbox{\ \ \ \ \ \ \ \ \ \,} endif
        \vspace{0.1cm}

\ \ \ \ $[3.4]$ $x_{\lambda_{k,0}} \, := \, x_{k-1}^\delta \, - \,
                    \lambda_{k,0} \, M_{\lambda_{k,0}} \, A^*(A x_{k-1}^\delta - y^\delta)$;
        \vspace{0.1cm}

\ \ \ \ $[3.5]$ compute $G_k(\lambda_{k,0}) \, = \, \norm{ A x_{\lambda_{k,0}} - y^\delta}^2$;
        \vspace{0.1cm}

\ \ \ \ $[3.6]$ $j := 0$; \ $\omega_0 := 1$;
        \vspace{0.1cm}

\ \ \ \ $[3.7]$ while \ $\big( G_k(\lambda_{k,j}) \, > \, \theta_k^2 \big)$  \ do \\
\mbox{\hskip1.8cm} compute $DG_k(\lambda_{k,j}) \, = \, \big\ipl A^*(A x_{\lambda_{k,j}} - y^\delta) \, , \,
                   M_{\lambda_{k,j}} \, A^* (A x_{\lambda_{k,j}} - y^\delta) \big\ipr$;
                   \vspace{0.1cm}

\mbox{\hskip1.8cm} $j := j + 1$;
                   \vspace{0.1cm}

\mbox{\hskip1.8cm} $\lambda_{k,j} \, := \, \lambda_{k,j-1} - \omega_{j-1} \, G_k(\lambda_{k,j-1}) / DG_k(\lambda_{k,j-1})$;
                   \vspace{0.1cm}

\mbox{\hskip1.8cm} $M_{\lambda_{k,j}} \, := \, (I + \lambda_{k,j} A^*A)^{-1}$;
                   \vspace{0.1cm}

\mbox{\hskip1.8cm} $x_{\lambda_{k,j}} \, := \, x_{k-1}^\delta \, - \,
                   \lambda_{k,j} \, M_{\lambda_{k,j}} \, A^*(A x_{k-1}^\delta - y^\delta)$;
                   \vspace{0.1cm}

\mbox{\hskip1.8cm} compute $G_k(\lambda_{k,j}) \, = \, \norm{ A x_{\lambda_{k,j}} - y^\delta}^2$;
                   \vspace{0.1cm}

\mbox{\hskip1.8cm} update over-relaxation parameter \ $\omega_{j}$;

\mbox{\ \ \ \ \ \ \ \ \ \,} end of while [3.7]
        \vspace{0.1cm}

\ \ \ \ $[3.8]$ $x_{k}^\delta \, := \, x_{\lambda_{k,j}}$; \ \ $\lambda_k \, := \, \lambda_{k,j}$;
        \vspace{0.1cm}

\ \ \ \ end of while [3]
} }
\end{center} \vskip-0.5cm
\caption{rrNIT algorithm.} \label{alg:rrNIT-numer}
\end{algorithm}

\bibliographystyle{amsplain}
\bibliography{projected-iT}

\end{document}